\documentclass[final,onefignum,onetabnum]{siamart171218}

\usepackage{lipsum}
\usepackage{amsfonts}
\usepackage{graphicx}
\usepackage{epstopdf}
\usepackage{algorithmic}
\usepackage{amstext,amssymb}
\usepackage{float} % !ht command in the image
\usepackage{amsopn}
\usepackage{array} %special column option for tables
\usepackage[caption=false]{subfig}

\ifpdf
  \DeclareGraphicsExtensions{.eps,.pdf,.png,.jpg}
\else
  \DeclareGraphicsExtensions{.eps}
\fi

\newsiamremark{remark}{Remark}
\newsiamremark{hypothesis}{Hypothesis}
\crefname{hypothesis}{Hypothesis}{Hypotheses}
\newsiamthm{claim}{Claim}
\newsiamremark{example}{Example}

\headers{Local Flux Recovery}{C. He}

\title{Flux recovery for Cut finite element method and its application in a posteriori error estimation 
  \thanks{
This work was funded by the EPSRC grant EP/P01576X/1.
}}

\author{
	Daniela Capatina
	\thanks{LMAP \& CNRS UMR 5142, University of Pau, 64013 Pau, France 
	(\email{daniela.capatina@univ-pau.fr})
	}
\and 
	Cuiyu He
	\thanks{Department of Mathematics, University of Georgia, Athens, GA, 
	30605
  	(\email{cuiyu.he@uga.edu})
	}
	}

\def\IR{\mathbb R}

\newcommand{\bfn}{\boldsymbol n}

\newcommand{\bfx}{\boldsymbol x}

\newcommand{\bfp}{\boldsymbol p}

\newcommand{\bn}{\boldsymbol n}

\newcommand{\bsigma}{\boldsymbol \sigma}

\newcommand{\bfnu}{\boldsymbol \nu}

\newcommand{\bftau}{\boldsymbol \tau}
\newcommand{\bfzeta}{\boldsymbol \zeta}

\newcommand{\cE}{\mathcal{E}}
\newcommand{\cN}{\mathcal{N}}

\newcommand{\tn}{|\mspace{-1mu}|\mspace{-1mu}|}

\newcommand{\Og}{\Omega}

\numberwithin{equation}{section}

\newcommand{\jump}[1]{[\![#1]\!]}

\newcommand{\tnorm}[1]{\tn#1\tn}

\newcommand{\divvr}{ \nabla \cdot}

\newcommand{\cT}{\mathcal{T}}

\ifpdf
\hypersetup{
  pdftitle={An Example Article},
  pdfauthor={D. Doe, P. T. Frank, and J. E. Smith}
}
\fi

\begin{document}

\maketitle

% REQUIRED
\begin{abstract}
In this article, we aim to recover locally conservative and $H(div)$ conforming fluxes for the linear  Cut Finite Element Solution with Nitsche's method for Poisson problems with Dirichlet boundary condition.
The computation of the conservative flux in the Raviart-Thomas space is completely local and does not require to solve any mixed problem. 
The $L^2$-norm of the difference between the numerical flux and the recovered flux can then be used as a posteriori error estimator in the adaptive mesh refinement procedure.
Theoretically we are able to prove the global reliability and local efficiency. The theoretical results are verified in the numerical results. Moreover, in the numerical results we also observe optimal convergence rate for the flux error.
\end{abstract}

% REQUIRED

% REQUIRED
\section{Introduction}
Cut finite element method (CutFEM) may be regarded as a fictitious domain method.
The finite element fictitious domain method was introduced in
\cite{GP92} as an approach to simplify the meshing problem and then further improved in \cite{BIM05,HR09, BH10a,BE84, HH02}. The challenges for such methods are typically that if the mesh is cut in an unfavorable way
 the system can be ill-conditioned and accuracy can be lost. Several
 approaches have been proposed to handle this problem, all based on
 the idea of extending the stability of the solution in the bulk up to
 the boundary. In \cite{HR09}, boundary fluxes are evaluated using the
 gradient extended from internal elements. In \cite{JL13,HWX17,BE18, BVM18}, agglomeration of adjacent
 elements is used. Finally, in \cite{Bu10,BH12}, the weakly consistent ghost penalty term
 was proposed that serves purpose. 
 
The purpose of this paper is to design and analyze a locally conservative flux in the Raviart-Thomas space of order $0$ and $1$ for the linear CutFEM.
We will base our discussion on the approximation of Poisson's problem 
\cite{BH12} in two dimensions. In addition, the CutFEM uses Nitsche's method \cite{Nit71} to impose
Dirichlet boundary conditions and a ghost penalty term \cite{Bu10} to 
enhance stability in the boundary zone. 

%\textcolor{red}{Recommended Journal: IMA Journal of numerical analysis, M2AN(France), M3AS(Italian).}
%For higher order polynomial
% approximation, the geometry approximation as well must be of higher
% order. To alleviate the
%integration problem resulting from the elements cut by curved
%boundaries, isoparametric techniques \cite{Lehr17} and so called 
%boundary value correction (or shifted boundary) techniques have been proposed 
%\cite{BHL18, BBCL18, MS18}. For both cases,
%optimal order a priori error estimates for arbitrary order of the
%polynomial approximation have been derived. 

One important application for the flux recovery is that the $L^2$-norm of the difference between the numerical flux and the recovered flux can be used in the a posteriori error estimation.
 The main motivation for studying the a posteriori  error estimation is the application of adaptive mesh refinement (AMR) procedure. It is well known that AMR is extremely useful for problems with singularities, discontinuities, sharp derivatives etc.. And it has been extensively studied in the last several decades, see e.g., \cite{verfurth1994posteriori, ainsworth2011posteriori}. 
  In \cite{EburmanChe-2020}, a residual based a posterior error estimator for the CutFEM method was studied. One drawback of residual based error estimation is that its 
  %global reliability control is not as accurate whose 
  reliability constants are unknown and usually not polynomial-robust and problem dependent. On the other hand, it is well known that the difference between the numerical flux and a locally conservative (equilibrate) flux automatically yields an upper bound for the true energy error with a reliability constant being exactly $1$. 

Thanks to the sharp reliability, equilibrate flux recovery has been extensively studied for various finite element methods on fitted meshes in the last decade. It is well known that for discontinuous Galerkin methods, a locally element-wise equilibrate flux can be easily obtained thanks to the fact that the test functions are completely local  \cite{bastian2003,ern2007,Ai:07b, BFH:14,becker2016local}. For nonconforming finite element methods of odd order, a local element--wise construction can be also easily obtained by taking advantage of the nonconforming local basis functions \cite{marini1985,he2020generalized}. For second order nonconforming finite element method, an explicit construction is designed in \cite{kim2012flux}.
However, for nonconforming methods of general even orders and conforming finite element methods, local element-wise (explicit) flux recovery is not straightforward and usually local problems on star patches need to be solved \cite{Ai:08,ErnVo2015,becker2016local}. In \cite{odsaeter2017}, a conservative flux is obtained by adding a piecewise constant correction through minimizing a weighted global $L^2$-norm. 
%In this paper, our aim is to design and analysis a locally conservative flux for CutFEM combining with Nitsche's method for Dirichlet boundary. 

We note that the method introduced in \cite{becker2016local} is applicable to various finite element methods, and furthermore, designed in a framework that fully takes advantages of the local basis functions for each method. For conforming finite elements on fitted meshes, this method only requires solving local problems that do not involve any hybrid mixed problem which are required e.g., in \cite{braess2009,cai2012robust,ErnVo2015}. In this work, we use a similar approach for CutFEM. For the interior elements not cut by the boundary, the recovered flux is similar to that in \cite{becker2016local} although it needs extra treatment for the ghost penalty term. 
Unlike the fitted methods whose mesh is an exact
 partition of the computational domain, one major challenge for cut
 finite elements is that the  domain cuts the background
 mesh in an arbitrary fashion. 
For cut elements, the recovery of the flux becomes complex due to non-standard terms in the variational formulation and partial intersections between the mesh and the domain, and, therefore, needs to be carefully designed to avoid artificial error. We divide boundary elements into two types, i.e., the set of boundary elements that are fitted and not fitted to the mesh, and design the flux differently.

Contrary to the classical approach, we consider the domain with non-polygonal boundary.
In order to achieve the same accuracy as the classical fitted method, 
boundary geometry and boundary data for CutFEM need to be approximated to similar
accuracy. It is therefore important in this context to derive error
estimators that are able to integrate both the discretization error of the
method and the discretization error of the geometry.
In this work, we approximate the physical geometry by a piecewise
affine polygonal domain. In \cite{EburmanChe-2020}, a boundary correction error was separated that particularly estimates the geometry approximation error. The computation of this term, however, is not trivial and requires the construction of a sub mesh. Nevertheless, numerical results have shown that such error is not necessary to compute since the boundary approximation error can already be captured by the residual based error estimator in \cite{EburmanChe-2020}. In this paper, we also discard such term and numerical results also confirm that our error estimator is able to catch both the boundary approximation errors and the discretization error due to the numerical method.

%Other uses for conservative flux recovery for flow problem, porous media. More interested in physical point of view, flux is the quantity of interest for linear elasticity. Porous media is interesting, efficient. 
%Numerical results shows optimal convergence rate. 

%\textcolor{black}{Including the boundary approximation error in the
%  analysis, our aim is to effectively and accurately identify the singularities and other types of lower regularities both inside the domain and on the boundary.}
%Another approach for
%the handling of geometric singularities in the CutFEM framework was
%recently proposed in \cite{JLL18}.

Although we herein restrict the discussion to the case of piecewise
affine approximation spaces from \cite{BH12}, we believe that the ideas introduced can
be extended for instance to the high order case discussed in \cite{BHL18} as regards the CutFEM method, and in \cite{becker2016local} as regards the flux reconstruction in two dimensions.
For other works treating a posteriori error estimation and cut cell techniques we refer to \cite{estep2011posteriori}, where a finite volume method was considered, and \cite{sun2020implicit} and \cite{di2019dual} where cut cell methods were applied.

This paper is organized as follows. In \cref{sec:2}, the model problem and the CutFEM are introduced.
In \cref{mixedProblem}, we design the local conservative flux by introducing an auxiliary mixed method and establish its well-posedness. 
In \cref{posteriori}, we apply the conservative flux in the a posteriori error estimation and establish its reliability and efficiency.
%In \cref{remarks}, we add some remarks for the programming techniques.
%In \cref{sec:3} we introduce the error estimator and prove its global reliability. The local efficiency is proved in  \cref{sec:4}.
%In \cref{sec:5}, we propose a method to estimate the boundary correction error. 
Finally, we show the  results of several numerical experiments in \cref{sec:6}.

\section{Model problem and the Cut Finite Element Method}\label{sec:2}
\subsection{The continuous problem}
Let $\Omega$ be a domain in $\mathbb{R}^d$ ($d=2$) with Lipschitz continuous,
piecewise smooth boundary $\partial \Omega$ with exterior unit normal $\bfn$. 
We consider the problem: find $u:\Omega \rightarrow \IR$ such that
\begin{equation}
\begin{split}
	-\Delta u &= f \qquad \mbox{ in } \Og,\\
	u &= g \qquad \mbox{on } \partial \Og,
\end{split}
\end{equation}
where $f\in L^2(\Omega)$ and $g\in H^{1/2}(\partial \Og)$.
For the sake of simplicity, we only consider the model problem of Laplacian operator with Dirichlet boundary conditions. However, the technique could be generalized to other boundary conditions and more complex elliptic operators.

Define the spaces
\[
	H^1_{g}(\Og) = \{v \in H^1(\Og): v = g \mbox{ on } \partial \Og\} \; \mbox{and} \;
	H^1_{0}(\Og) = \{v \in H^1(\Og): v = 0 \mbox{ on } \partial \Og\}.
\]
Then the weak formulation for this problem renders to find $u \in H^1_{g}(\Og)$ such that
\[
	a(u,v) = (f,v)_{\Og}, \quad \forall \,v \in H_{0}^1(\Og),
\]
where $a(u,v) = (\nabla u, \nabla v)_{\Og}$.
It follows from the Lax-Milgram lemma that there 
exists a unique solution $u \in H_{g}^1(\Omega)$ to this problem. 
%We also 
%have the  following elliptic regularity estimate
%\begin{equation}\label{eq:ellipticregularity}
%\|u\|_{H^{s+2}(\Omega)} \lesssim \|f\|_{H^s(\Omega)}, \qquad 
%-1 \leq s \leq s_0
%\end{equation} 
%for some $s_0\geq 1/2$ depending on the domain.

\subsection{The mesh, discrete domain, and finite element spaces}
Assume that
$\partial \Omega$ is composed of a finite number of
  smooth surfaces $\Gamma_i$, such that $\partial \Omega = \underset{i}{\cup} \bar{\Gamma}_i$. 
We let $\rho$ 
be the signed distance function such that
\[
	\rho(x)
	\begin{cases} 
		< 0 &\mbox{if } x \in \Og,\\
		= 0 &\mbox{if } x \in \partial \Og,\\
		>0 &\mbox{if } x \in  \bar\Og^c,
	\end{cases}
\]
where  $\bar\Og^c$ is the complement of the closure of $\Og$.
We define
$U_\delta(\partial \Omega), \delta>0$,
be the tubular neighborhood $\{\bfx \in \IR^d : |\rho(\bfx)| < \delta\}$ 
of $\partial \Omega$.  Choose $\Omega_0 \subset \IR^d$ be 
the background domain (see e.g., the square outline of the entire mesh in \cref{Fig:domain-example}) such 
that it is polygonal, $\Og \subset \Og_0$ and 
$U_{\delta_0}(\partial \Omega)\subset \Omega_0$ 
where $\delta_0$ is chosen such that $\rho$ is well defined in $U_{\delta_0}(\partial\Omega)$. 
Let $\cT_{0,h}$ 
be a partition of $\Omega_0$ into shape 
regular triangles or tetrahedra (see e.g., the mesh in \cref{Fig:domain-example}). Note that this setting allows meshes with locally  
dense refinement.

\begin{center}
\begin{figure}\label{Fig:domain-example}
 \includegraphics[width=0.80\textwidth]{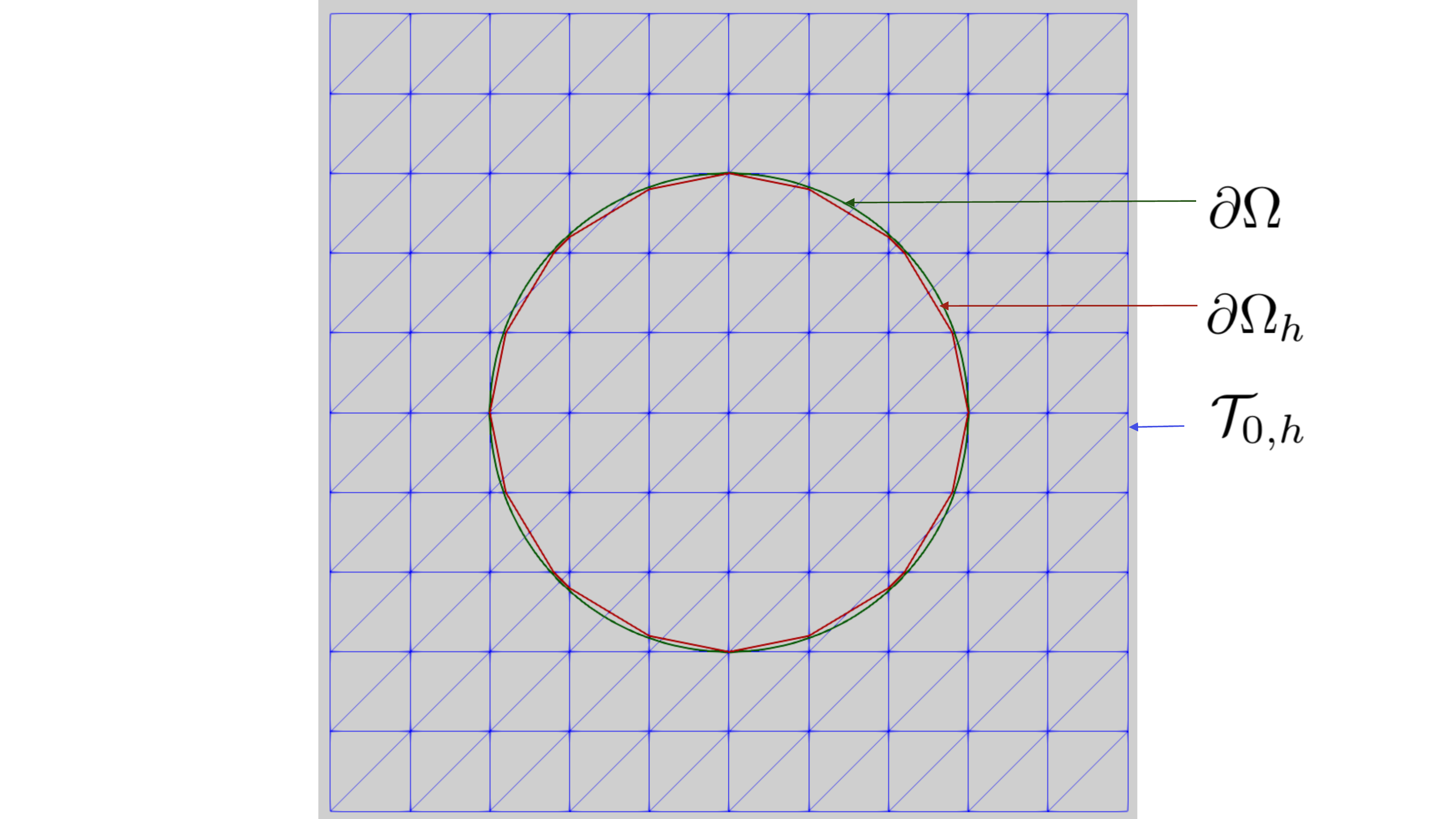}
 \caption{An example for the background mesh $\cT_{0,h}$,  $\partial \Og$ and $\partial \Og_h$. $\Og_0$ is the entire square domain.}
\end{figure}
\end{center}

Given a subset $\omega$ of $\Omega_0$, let 
$\cT_h(\omega)$ be the sub-mesh defined by
\begin{equation}
\cT_{h}(\omega) = \{K \in \cT_{0,h} : {K} \cap 
{\omega}
\neq \emptyset \},
\end{equation}
i.e., the sub-mesh consisting of elements that have non-zero intersection with
${\omega}$, and let 
\begin{equation}
\triangle_{h}(\omega)= \bigcup\limits_{K \in \cT_{h}(\omega)}K, %\cup_{K \in \cT_{h}(\omega)} K
\end{equation}
which is the union of all elements in $\cT_h(\omega)$.

For each $\cT_{0,h}$, let $\Omega_h$ (see \cref{Fig:domain-example} for an example) be a polygonal domain approximating $\Omega$. We assume that
$\partial \Omega_h \subset U_{\delta_0}(\partial \Omega)$, i.e.,  $\partial \Omega_h$ is within the distance of $\delta_0$ to $\partial \Omega$. Moreover, we also require that the maximum distance between the two domains is small enough so that $\Og_h$ is a sufficiently good approximation to $\Og$. More details will be given later.

Let the active mesh be defined by
\begin{equation}
\cT_{h} := \cT_h( \Omega_h) %(\mbox{why not just} $\O_h$?)
\end{equation}
i.e., the sub-mesh consisting of elements that intersect 
$\Omega_h$, and let 
\begin{equation}
\triangle_h := \triangle_h( \Omega_h).
\end{equation}
Since $\partial \Og_h$
cut the active mesh $\cT_h$ in an arbitrary fashion, we denote by $\cT_h^b$ the set of elements that are ``cut" by $\partial \Og_h$, i.e.,
\[
\cT_h^b = \{ K \in \cT_{0,h}: K \cap \partial \Og_h \neq \emptyset\} \subset \cT_h.
\]
We further assume that $\Omega_h$ is constructed in such a 
way that for each $K  \in \cT_{h}^b$, the intersection $K\cap \partial \Omega_h$ is a subset of 
a $d-1$ dimensional hyperplane, i.e., a line segment in two dimensions. Under the assumptions on $\partial
  \Omega$, and if we also assume that $h$ is sufficiently small, then for any element
  $K \in \cT_h^b$ there
exists an element $K' \in \cT_h\setminus \cT_h^b$ such that
$\mbox{dist}(K,K') = O(h_K)$ where $O(\cdot)$ denotes the Landau
big-$O$.

Also we denote by $\cE$ the set of all facets of $\cT_h$, and by $\cE_I$ and $\cE_{\partial}$ the set of all interior and boundary facets with respect to $\cT_h$, respectively. It is obvious that $\cE_I \cup \cE_{\partial} = \cE$.
For each $F \in \cE$ denote by $\bfn_F$ an unit vector normal to $F$ and by $h_F$ (or $|F|$) the diameter of $F$. If $F \in \cE_{\partial}$, then $\bfn_F$ is the outward unit normal vector.
For each $K \in \cT_h$ denote by $h_K$ the diameter of $K$ and by $\cE_K$ the set of all facets of $K$.

On the boundary $\partial \Og_h$, let $\bfn_h$ be the outer normal to $\partial \Og_h$.
For each $\Omega_h$, we assume that, for $\delta_0$ small enough,
there exist a vector function 
$\bfnu_h:\partial \Omega_h \rightarrow \mathbb{R}^d$, $|\bfnu_h|=1$, and $\varrho_h:\partial \Omega_h \rightarrow \mathbb{R}$, %on $U_{\delta_0}(\partial \Omega)$, 
such that the function,
$\bfp_h(\bfx,\varsigma):=\bfx + \varsigma \bfnu_h(\bfx)$, is well
defined and satisfies
$\bfp_h(\bfx,\varrho_h(\bfx)) \in \partial \Omega$ for all $\bfx
\in \partial \Omega_h$. Here $\varrho_h$ is the distance from the
approximate to the physical boundary in the direction $\bfnu_h$. The
mapping $\bfp_h$ allows us to impose the boundary data defined on the
physical boundary on the approximate boundary.
The existence of the vector-valued function $\bfnu_h$ is known to hold
on Lipschitz domains, see Grisvard \cite{Gris11}.
We further assume that 
\[
\bfp_h(\bfx,\varsigma)
\in U_{\delta_0}(\partial\Omega), \quad \forall \bfx \in \partial \Omega_h, \;  0 \le \varsigma \le \varrho_h(\bfx).
\]
For convenience,
we will drop the second argument, $\varsigma$, of $\bfp_h$ below whenever it takes the value $\varrho_h(\bfx)$ and then $\bfp_h$ denotes the map $\bfp_h:\partial \Omega_h \rightarrow \partial \Omega$.
Moreover, we assume that the following assumption is satisfied
\begin{equation}\label{eq:geomassum-a}
%\delta_{h,K} := 
\| \varrho_h \|_{L^\infty(\partial \Omega_h \cap K)} \le  O(h_K) \quad
\forall \, K \in \cT_h^b.
\end{equation}
The above assumption immediately implies that $\partial \Og_h \cap K$ is within the distance of $O(h_K)$ of $\partial \Og$,
i.e.,
\begin{equation}
\| \varrho \|_{L^\infty(\partial \Omega_h \cap K)} \le  O(h_K) \quad
\forall \, K \in \cT_h^b.
\end{equation}
This assumption is necessary for the constant in the a posteriori
error estimates to be independent of the geometry/mesh
configuration see \cite{EburmanChe-2020}. It is however not enough to guarantee the optimal a priori
error estimates, which requires $\| \varrho \|_{L^\infty(\partial \Omega_h \cap K)} \le O(h_K^2)$, 
and $\| \bfn \circ \bfp_h - \bfn_h \|_{L^\infty(\partial \Omega_h \cap K)} \le O(h_K)$, 
see \cite{BHL18}. 
It is noted that in theory we can assume neither $\Omega_h \subset \Omega$ nor  $\Omega \subset
\Omega_h$. However, for simplicity, we assume $\Og \subset \Og_h$. 
This will help skip the analysis for the a posteriori error estimation on the part of data approximation error of $f$. Such error can be eventually ignored in the computation since it is of higher order.

\subsection{The Cut Finite Element Method} 
In this subsection we recall the CutFEM introduced in \cite{BH12}. We begin with some 
necessary notation. 
For any $F\in \cE_I$, let $K_F^+$ and $K_F^-$ be those two elements sharing $F$ as a common facet such that the outer normal of $K_F^+$ coincides with $\bn_F$.
For any discontinuous function $v$, define the jump of $v$ across the facet $F$ by 
\[
\jump{v}|_F := v^+_F - v^-_F \quad \mbox{and} \quad 
v_F^{\pm}(\bfx)
= \lim \limits_{s \rightarrow 0^+} 
v(\bfx \mp s \bfn_F).
\]
The set of facets associated with cut elements is defined by 
\[ 
\cE_g := \{ F \in \cE_I \,:\, (K_F^+ \cup K_F^-) \cap \partial \Og_h \neq \emptyset\}.
\]
The index $g$ above refers to the ghost penalization, defined on every $F\in \cE_g$.

For each $K \in \cT_h$, we define a sign function $\mathfrak{s}_K$ defined on $\cE_K$ such that
\[
	\mathfrak{s}_K(F) = \begin{cases}
	1 \qquad \mbox{if } \bn_F = \bn_K|_F,\\
	-1 \qquad \mbox{if } \bn_F = -\bn_K|_F.
	\end{cases}
\]
The conforming linear finite element space is then defined as
\begin{equation}
CG_{h} :=\{ v \in H^1(\triangle_h): v|_K \in \mathbb{P}_1(K) \quad \forall \, K \in \cT_{h} \}.
\end{equation}
We also define the following forms:

\begin{equation}\label{forms}
\begin{split}
a_0(w,v) &:= (\nabla w,\nabla v)_{\Omega_h} 
- \left<\partial_{\bfn_h} w,v\right>_{\partial \Og_h} 
- \left<w, \partial_{\bfn_h} v\right>_{\partial \Og_h} 
+   \sum_{K \in \cT_h^b} \dfrac{\beta}{h_K}\left< w,v\right>_{\Gamma_K},\\
j_h(w,v) &:= \gamma\sum_{F \in \cE_g}  h_F \left< \jump{\partial_{\bfn_F} w},\jump{\partial_{\bfn_F} v}\right>_F,\\
a_h(w,v)& := a_0(w,v) + j_h(w,v), 
\\
l_h(v) &:= (f,v)_{\Omega_h} 
- \left<g_h , \partial_{\bfn_h} v\right>_{\partial \Og_h} 
+ \sum_{K \in \cT_h^b} \dfrac{\beta}{h_K}\left<g_h,v\right>_{\Gamma_K} ,
\end{split}
\end{equation}
where $\Gamma_K = K \cap \partial \Og_h$, $\partial_{\bfn_h} := \bfn_h \cdot \nabla $,  $\gamma$ and $\beta$ are positive constants, $g_{h} $ is an approximation of $g$ defined on $\partial \Og_h$. A natural choice is that $g_h(\bfx) = g \circ \bfp_h$. 
Note that we only assumed that $\Og \subset \Og_h$. In $\Omega_h \setminus
  \Omega$ where $f$ is not originally defined,  $f\vert_{\Omega_h \setminus
  \Omega}$ is defined to be an appropriate extension. 
 
\begin{remark}
The stabilizing term $j_{h}(w,v)$, which is the so-called ghost penalty term, is introduced to extend the coercivity of $a_h(\cdot,\cdot)$ to all of $\triangle_h$, see  \cite{Bu10,MasLarLog14}. Thanks to this property, one may prove that the condition number of the linear system is uniformly bounded regardless of the arbitrary boundary-mesh intersection.
\end{remark}

\begin{remark}
In order to guarantee the coercivity of the bilinear form $a_h(\cdot,\cdot)$ in \cref{forms}, $\beta$ has to be chosen large enough.
\end{remark}

The finite element method is then to find  $u_h \in CG_h$ such that 
\begin{equation}\label{eq:fem}
a_h(u_h, v) = l_h(v) \quad \forall \,v \in CG_h
\end{equation}
where $a_h$ and $l_h$ are defined in \cref{forms}. 

For $v \in H^1(\triangle_h)$, define the continuous and discrete energy norms
respectively by
\begin{align}\label{def:ene_norm}
\tn v \tn^2 &:= 
\|\nabla v\|^2_{\Omega} 
+\|\partial_{\bfn}v \|^2_{H^{-1/2}(\partial \Omega)} 
+ \|h^{-\frac12} v\|^2_{\partial \Omega}
\end{align}
and
\begin{align}\label{def:ene_norm_h}
\tn v \tn_h^2 &:= 
\|\nabla v\|^2_{\Omega_h} 
+\sum_{K \in \cT_h^b} h_K \|\partial_{\bfn_h} v \|^2_{\Gamma_K} 
+ \sum_{K \in \cT_h^b}h_K^{-1}\|v\|^2_{\Gamma_K}+j_h(v,v).
\end{align}

In \cref{def:ene_norm}, $h$ denotes the piecewise constant mesh size function.
From \cite{Bu10}, we have for $\beta$ sufficiently large the following coercivity result,
\begin{equation}\label{coecivity}
	 a_h(v, v) \ge C \tnorm{v}_h^2 \quad \forall v \in CG_h,%double check
\end{equation}
which, together with the uniform $\tn \cdot \tn_h$-continuity of $a_h(\cdot,\cdot)$ on $CG_h$,  implies that \cref{eq:fem} has a unique solution.

\subsection{Some important inequalities}
Below we list the well known trace and inverse inequalities \cite[Section 1.4.3]{DiPE12},
\begin{equation}\label{eq:standard_trace}
\|v\|_{\partial K} \lesssim h_K^{-1/2}\|v\|_{K} + h_K^{1/2}\|\nabla
v\|_K \quad \forall \, v \in H^1(K), \quad \forall \, K \in \cT_h,
\end{equation}
\begin{equation}\label{eq:discrete_trace}
h_K^{-\frac12} \|v_h\|_{\partial K} + \|\nabla v_h\|_K \lesssim h_K^{-1}\|v_h\|_{K} 
\quad \forall \, v_h \in \mathbb{P}_1(K), \quad \forall \, K \in \cT_h.
\end{equation}
The following irregular trace inequality can be found in \cite{HH02}
\begin{equation}\label{eq:boundary_trace}
	\|v_h\|_{\Gamma_K} \lesssim h_K^{-1/2}\|v_h\|_{K} +
        h_K^{1/2}\|\nabla v_h\|_K \quad \forall \, v_h \in \mathbb{P}_1(K),
        \quad \forall \, K \in \cT_h^b,
\end{equation}
where the hidden constant is independent of the boundary-mesh intersection.
 Here and below we use the notation 
$\lesssim$ to denote less or equal up to a generic constant that is independent of the
mesh-geometry configuration.

In the following Lemma, we also provide a Poincar\'e-type inequality for the boundary elements.
\begin{lemma}
Let $v \in H_{0}^1(\Og)$. Then for any $K$ such that $K \cap \partial \Og  \neq \emptyset$ or $K \cap \partial \Og_h  \neq \emptyset$,
%$\{K \in \cT_h^b, K \cap \partial \Og_h \neq \emptyset \}$ or 
%$\{K \cap (\Og \setminus \Og_h) \neq \emptyset, K \cap \Gamma_D \neq \emptyset \}$, 
there exists a local convex neighborhood $\mathcal{S}_K$ of $K$ such that $v$ vanishes 
on a non-empty subset of $\partial S_K$ and
\begin{equation}\label{poincare}
	\|v\|_{K} \lesssim h_K\| \nabla v\|_{\mathcal{S}_K},
\end{equation}
where we defined $v$ outside $\Og$ using the trivial extension $v\vert_{\Og_0 \setminus \Og} = 0$.
\end{lemma}
The proof of the lemma can be found in \cite{EburmanChe-2020}.

\section{Mixed formulation}\label{mixedProblem}
In this section, we introduce an auxiliary mixed formulation for the cut finite element method. The aim is not to solve the global mixed problem. Instead, our goal is to establish the connections between the mixed and CutFEM formulation, and, with the help of those connections, to locally recover a conservative flux in the $H(div)$ space. The idea generates from \cite{becker2016local} in which classical finite element methods with fitted meshes were studied.

Define  the discontinuous finite element space on $\triangle_h$ by
\begin{equation}
DG_{h} :=\{ v \in L^2(\triangle_h): v|_K \in \mathbb{P}_1(K) \; \forall \, K \in \cT_{h}
\}.
\end{equation} 
Also define the  average operator:
\[
	\{w\} = \begin{cases} 0.5 ( w_F^+ +  w_F^-), & F \in \cE_I,\\
	w ,& F \in  \cE_{\partial}.
	\end{cases}
 \]

We also denote by $\cN$ the set of all vertices in $\cT_h$,  and by $\cN_I$ and $\cN_{\partial}$ the sets of all interior and boundary vertices on $\cT_h$. Obviously, we have $\cN_I \cap \cN_{\partial} = \cN$.
For each $N \in \cN$, define $\cT_N$ and $\cE_N $ be the sets of all elements and all facets sharing $N$ as a vertex, respectively. For each $N \in \cN$, we define a signed function $\mathfrak{s}_N$ on $\cE_N$ such that
$\mathfrak{s}_N (F) = 1$ if $\bn_F$ is oriented counter-clockwise in $\cT_N$, otherwise $\mathfrak{s}_N (F) = -1$.
We further define the following space:

\[
	M_h =\left \{\mu \in L^2(\cE_I) : \mu|_F \in \mathbb{P}_1(F) \; \forall \, F \in \cE_I, 
\sum_{F \in \cE_N} \mathfrak{s}_N(F) h_F \mu|_F(N)  =0 \; \forall N \in \cN_I \right\}.
\]
%\textcolor{red}{Is $h_F$ here also correct in 3D?}
Here  $\mathbb{P}_1(F)$ is the space of linear functions defined on $F \in \cE$.
For the sake of brevity, we use the following notations for piecewise integration,

\[
	\int_{\cT_h} = \sum_{K\in \cT_h} \int_K, \quad \int_{\tilde \cE} = \sum_{ F \in \tilde\cE } \int_F,
\]
where $\tilde \cE$ is a subset of $\cE$.
We now define semi-norms and norms on the discrete spaces defined above:
\begin{equation}
\begin{split}
	|v|_{1,h} =& \left( \int_{\cT_h} |\nabla v|^2\right)^{1/2}, \quad \forall \,v \in DG_h,\\
	\tnorm{v}_{h,*} =&\left( |v|_{1,h}^2 +\sum_{K \in \cT_h^b} h_K \|\partial_{\bfn_h} v \|^2_{\Gamma_K} 
+ \sum_{K \in \cT_h^b}h_K^{-1}\|v\|^2_{\Gamma_K}+ \int_{\cE_I} h_F^{-1} \jump{v}^2 \,ds \right)^{1/2}, \forall v \in DG_h,\\ 
	 \| \mu\|_{M_h} =  &\left( \int_{\cE_I} h_F \mu ^2 \,ds \right)^{1/2}, \quad \forall \,\mu \in M_h.
\end{split}
\end{equation}

The auxiliary mixed formulation is defined as follows:
find $(u_h, \theta_h) \in DG_h \times M_h$ such that
\begin{equation}\label{mixed-formulation}
	\begin{split}
	&\tilde a_h (u_h, w_h) + b(\theta_h, w_h)= l_h(w_h), \quad \forall w_h \in DG_h, \\
	&b(\mu_h, u_h) = 0, \quad \forall \mu_h \in M_h.
	\end{split}
\end{equation}
where
\begin{equation}
\begin{split}
	\tilde{a}_h(v,w) =& a_h(v,w) - \left< \{ \partial_{\bn_F} v\}, \jump{w} \right>_{\cE_{I} \cap \Omega_h}- \left< \{ \partial_{\bn_F} w\}, \jump{v} \right>_{\cE^{I} \cap \Omega_h},\\
	b(\mu, v) =& \sum_{F \in \cE_I}\dfrac{h_F}{2} \sum_{N \in \cN\cap F} \mu_F(N) \jump{v}(N)
	\approx \int_{\cE_I} \mu \jump{v} \,ds, % \left< \mu, \jump{v} \right>_{\cE^{I}}.
\end{split}
\end{equation}
where $\mu_F = \mu|_F$ and $\cN\cap F$ is the set of vertices of $F$.
It is easy to check that the discrete kernel of $b(\cdot)$ coincides exactly with the CutFEM space, i.e.,
\[
	\mbox{ker}(b) := \{w_h: b(\mu_h, w_h)=0 \quad \forall \, \mu_h \in M_h\} = CG_h.
\]
Indeed, any function $w_h$ in $CG_h$ satisfies $\jump{w_h}|_F =0$ for any $F\in \cE_I$, and hence belongs to  $\mbox{ker}(b) $.  Reciprocally, for any $w_h$ in $\mbox{ker}(b)$, we can choose $\mu_h\in M_h$ defined by $(\mu_h)_F =h_F^{-1} \jump{w_h}|_F $ for any $F\in \cE_I$, which yields $\jump{w_h}|_F  =0$ and hence, $w_h \in CG_h$.
It is then obvious to see that the solution $u_h$ for \cref{eq:fem} coincides with the solution for \cref{mixed-formulation} if it is well defined. 

\subsection{Well-posedness of the mixed finite element approximation}

\begin{lemma}[continuity]\label{continuity}
We have the following continuity results for the bilinear forms:
\begin{equation*}
\begin{split}
	b(\mu, v) &\lesssim  \| \mu\|_{M_h} \tnorm{v}_{h,*}, \quad \forall \mu \in M_h, \forall v \in DG_h,\\
	\tilde{a}_h(v,w) &\lesssim  \tnorm{v}_{h,*} \tnorm{w}_{h,*}, \quad \forall v,w \in DG_h.\\
	\end{split}
\end{equation*}
\end{lemma}
\begin{proof}
The proof of the first assertion is trivial by Cauchy-Schwartz and the definitions of the norms:
\[
	b(\mu, v) \lesssim \sum_{F \in \cE_I} \| h_F^{1/2}\mu\|_{F} \|h_F^{-1/2} \jump{v}\|_{F} \le 
	\|\mu\|_{M_h}\tnorm{v}_{h,*}.
\]
As regards the second one, we first note that $j_h(v,v)\lesssim \int_{\cT_h} |\nabla v|^2$  for any $v\in DG_h$, so we clearly have that $\tnorm{\cdot}_{h} \lesssim \tnorm{\cdot}_{h,*} $ on $DG_h$, hence the continuity of  $a_h(\cdot,\cdot)$. The remaining terms of  $\tilde{a}_h(\cdot,\cdot)$ are bounded as  follows:
\begin{equation*}
\begin{split}
\vert\left< \{ \partial_{\bn_F} v\}, \jump{w} \right>_{\cE_I\cap \Omega_h}\vert &\le \left( \sum_{F\in \cE_I} h_F \| \{ \partial_{\bn_F} v\}\|_{F}^2\right)^{1/2}  \left( \sum_{F\in \cE_I} h_F^{-1}\|\jump{w}\|_{F}^2\right)^{1/2}\\
&\lesssim \left(  \int_{\cT_h} |\nabla v|^2\right)^{1/2}  \left( \int_{\cE_I} h_F^{-1} \jump{w}^2 \right)^{1/2}\qquad \forall v,w\in DG_h,
\end{split}
\end{equation*}
which ensures the continuity of $\tilde{a}_h(\cdot,\cdot)$.
\end{proof}

\begin{lemma}[inf-sup condition]
We also have the inf-sup result for the bilinear form: 
\begin{equation}\label{inf-sup}
	\inf_{\mu \in M_h} \sup_{v \in DG_h} \dfrac{b(\mu, v)}{\tnorm{v}_{h,*}   \| \mu\|_{M_h}} \ge C,
\end{equation}
where the constant is independent of the mesh size and domain-mesh intersection.
\end{lemma}
\begin{proof}
We prove by construction. To define a linear function $v \in DG_h$, it suffices to define 
$v_K(N):= v|_K(N)$ for all $K \in \cT_h$ and for all $N \in \cN_K$.
For each $N \in \cN_I$, we let $\{K_i\}_{i=1}^{n_N}$ be the clockwise oriented elements in $\cT_N$ where $n_N \ge 1$ is the number of elements in $\cT_N$. We also let $F_i = K_{i-1} \cap K_{i}, i=1, \cdots, n_N (K_{0} = K_{n_N})$. We let $v_{K_1}(N) =0$. The rest are defined such that
\begin{equation}\label{inf-sup-a}
v_{K_{i}}(N) - v_{K_{i-1}}(N) = h_{F_i} \mathfrak{s}_N (F_i)  \mu_{F_i}(N),\qquad  i=2, \cdots, n_N.
\end{equation}
It is easy to check that \cref{inf-sup-a} is compatible. Moreover, it is easy to check that 
\begin{equation}\label{inf-sup-b}
\jump{v}|_{F_i}(N) = h_{F_i}\mu_{F_i}(N), \qquad i=1, \cdots, n_N.
\end{equation}
For $N \in \cN_{\partial}$ such that $\cE_N \cap \cE_I\neq \emptyset$, we define $v_{K_i}, i=1, \cdots, n_N$ in the same way (where we assume that $K_1$ has a boundary facet). 
Note that there are only $n_N-1$ interior facets in this case. And we also have
\begin{equation}\label{inf-sup-c}
\jump{v}|_{F_i}(N) =h_{F_i} \mu_{F_i}(N), \qquad i=2, \cdots, n_N.
\end{equation}
Finally, for $N \in \cN_{\partial}$ such that $\cE_N \cap \cE_I =\emptyset$, we simply put $v_K(N)=0$.
Combining \cref{inf-sup-b} and \cref{inf-sup-c} gives that
 \begin{equation}\label{inf-sup-d}
\jump{v}|_F = h_F\mu_F \quad \forall \,F \in \cE_I.
\end{equation}
We then have that
 \begin{equation}\label{inf-sup-d1}
	b(\mu, v) = \sum_{F \in \cE_I}  \sum_{N \in \cN\cap F}\dfrac{h_F^2}{2}\mu_F(N)^2   \gtrsim  \| \mu\|_{M_h}^2.
\end{equation}
The last inequality follows from the equivalence of norms in a finite dimensional space.
Immediately from \cref{inf-sup-c}, we also have the following bound:
\begin{equation}\label{bound-v-k}
	\sum_{i=1}^{n_N} v_{K_i}(N)^2 \lesssim \sum_{F \in \cE_N \cap \cE_I} h_F^2 \mu_{F}(N)^2.
\end{equation}
By a direct computation, \cref{bound-v-k}, and norm equivalence in a finite dimensional space, we have
\begin{equation}\label{bound-v-1}
\begin{split}
	 \sum_{K \in \cT_h} \|\nabla v\|_{K}^2   &\lesssim  \sum_{K \in \cT_h} h_K\sum_{N \in \cN_K}v_K(N)^2 
	=  \sum_{N \in \cN} \sum_{K \in \cT_N}h_Kv_K(N)^2 \\
	&\lesssim \sum_{N \in \cN} \sum_{F \in \cE_N \cap \cE_I} h_F^2 \mu_F(N)^2
	\lesssim  \| \mu\|_{M_h}^2.
\end{split}
\end{equation}
Immediately, we also have, by using \cref{eq:boundary_trace} for the last estimate,
\begin{equation}\label{bound-v-2}
\begin{split}
	\int_{\cE_I} h_F^{-1} \jump{v}^2 \,ds
	&=\int_{\cE_I}  h_F \mu_F^2 \,ds = \| \mu\|_{M_h}^2,\\
	%j_h(v,v)&\lesssim   \sum_{K \in \cT_h} \|\nabla v\|_{K}^2 \lesssim \| \mu\|_{M_h}^2,\\
	\sum_{K \in \cT_h^b} \int_{\Gamma_K} h_K^{-1} v^2 \,ds   
	&\lesssim \sum_{K \in \cT_h^b}  h_K\sum_{N \in \cN_K}v_K(N)^2\lesssim  \| \mu\|_{M_h}^2.
\end{split}
\end{equation}
We also have, thanks to \cref{eq:boundary_trace}, 
\begin{equation}\label{bound-v-3}
 h_K \|\partial_{\bn_h} v\|^2_{\Gamma_K} \,ds\lesssim \|\nabla v\|_{K}^2,\qquad \forall K \in \cT_h^b.
\end{equation}
Note that the involved constants do not depend on the interface-mesh intersection.
Combing \cref{bound-v-1}, \cref{bound-v-2} and \cref{bound-v-3} yields
 \begin{equation}\label{inf-sup-e}
	\tnorm{v}_{h,*} \lesssim \| \mu\|_{M_h}.
 \end{equation}
 \cref{inf-sup} is then the direct consequence of \cref{inf-sup-d1} and \cref{inf-sup-e}. This completes the proof of the lemma.
\end{proof}

By \cref{coecivity}, we also have the uniform coercivity of $\tilde{a}_h(\cdot,\cdot)$ on $\mbox{ker}(b)=CG_h$ with respect to the norm $\tnorm{\cdot}_h $.

\begin{lemma}
	The mixed formulation \cref{mixed-formulation} is well posed. Moreover, the solution $u_h$ in 
	\cref{mixed-formulation}  coincides with the CutFEM solution of \cref{eq:fem}. 
\end{lemma}
\begin{proof}
The proof is standard and follows from the Babuska-Brezzi theorem.  We sketch the proof of the existence and uniqueness of the solution for the convenience of the readers.
It is enough to show the uniqueness of the solution of the linear square system. 
We prove by contradiction. Assume both $(u_h^1, \theta_h^1)$ and  $(u_h^2, \theta_h^2)$ are solutions to \cref{mixed-formulation}.
From the second equation in \cref{mixed-formulation}, we have that $(u_h^1 - u_h^2) \in CG_h$. Then we have
\[
	a_h(u_h^1 - u_h^2,w_h) =0 \qquad \forall \, w_h \in CG_h.
\]
By \cref{coecivity}, we have $(u_h^1 - u_h^2) \equiv 0$. 
Then we have $b(\theta_h^1 -\theta_h^2 , w_h)=0$ for all $w_h \in DG_h$. Finally by the inf-sup condition \cref{inf-sup}, we obtain $\theta_h^1 -\theta_h^2 \equiv 0$. Since the solution for  \cref{mixed-formulation} is unique, then the solution $u_h$ in \cref{mixed-formulation} must  coincide with the CutFEM solution of \cref{eq:fem}. 
%Thus, the a priori error estimate for $u-u_h$ follows from the analysis of the CutFEM formulation, see \cite{BHL18}.
%Thanks to the uniform inf-sup condition satisfied by $b(\cdot,\cdot)$,  one gets the following optimal error estimate for $\theta_h$:
%The consistency of the primal formulation  together with $\jump{u}=0$ on any $F\in\cE_I$ imply, after integration by parts, the key property: 
%\begin{equation*}
%a_h(u,v_h)=l_h(v_h),\quad \forall v_h\in D_h.
%\end{equation*}
%\begin{equation*}
%\tnorm{\theta_h} \lesssim \sup_{v_{h}\in D_h}\frac{b_h(\theta_h,v_{h})}{\tnorm{v_{h}}}=C \sup_{v_{h}\in D_h}\frac{a_h(u,v_{h})-a_h(u_h,v_h)}{\tnorm{v_{h}}}.
%\end{equation*}
This completes the proof of the lemma.
\end{proof}

\subsection{Local construction of $\theta_h$}
In this subsection, we aim to compute the solution $\theta_h$ to \cref{mixed-formulation} based on the solution $u_h$ of \cref{eq:fem} through solving local problems, following \cite{becker2016local}.
Firstly note that 
\begin{equation}\label{residual1}
	\begin{split}
	b(\theta_h, w) &= r(w):=l_h(w) - \tilde a_h(u_h, w) \quad \forall w \in DG_h.
	\end{split}
\end{equation}
Note that $r(w) =  0$ for any $w \in CG_h$.

Let $N \in \cN$ such that $\cE_N \cap \cE_I \neq \emptyset$.  We define $\theta_N \in M_h$ on $\cE_N \cap \cE_I$ such that, for each $K \in  \cT_N$,
\begin{equation}\label{local-construction}
	\begin{split}
	b(\theta_N, \lambda_N\chi_K) &= r(\lambda_N \chi_K)\\
	b(\theta_N, \lambda_M \chi_K) & = 0 \quad \mbox{if }  M\in \cN_K \mbox{ and }  M \neq N,
	\end{split}
\end{equation}
where $\lambda_M$ denotes the barycentric basis function corresponding to the vertex $M$. On $ \cE_I\setminus \cE_N $, we impose that $\theta_N$ is null.

Recall from the definition of $M_h$ that the condition $\theta_N \in M_h$ means that $\theta_N$ has to satisfy a constraint at any interior node. The last equation of \cref{local-construction} gives that $\theta_N|_{F}(M) =0$ for all $F\in \cE_N $ and $M\in F,\, M \neq N$, so the constraint is obviously satisfied at the node $M$, for $M\in \cN_I$. Hence,  $\theta_N$ only has to satisfy the constraint equation at the node $N$, for $N\in \cN_I$.

Also note that $ \sum \limits_{K \subset \cT_N}r(\lambda_N \chi_K)=r(\lambda_N)=0$. It is then easy to check that the system introduced by \cref{local-construction} is  compatible.

\begin{lemma}
 For any $N \in \cN$, the system \cref{local-construction} has a unique solution $\theta_N$ in $M_h$.
\end{lemma}

\begin{proof}
We first assume that $N \in \cN_I$.
Because of the compatibility condition, the system \cref{local-construction} a one-dimensional kernel. Indeed, let $\Psi_N $ defined on any  $F\in \cE_N $ by $\Psi_N|_{F}\in \mathbb P_1(F)$ and
\[
\begin{cases}\Psi_N|_{F}(N) =h_F^{-1} \mathfrak{s}_N(F),\\
\Psi_N|_{F}(M) =0, &\forall M\in F,\, M \neq N,
\end{cases}
\]
whereas on  $F\in\cE_I\setminus \cE_N$, we set $\Psi_N|_{F}\equiv 0$. 
%As for $\theta_N$, we get that $\Psi_N|_{F}=0$  for any $F\in \cE_N \setminus \cE_I$. 
It is then easy to check that for all $K \in \cT_N$,
\begin{equation}
	\begin{split}
	b(\Psi_N, \lambda_N\chi_K) &= 0,\\
	b(\Psi_N, \lambda_M \chi_K) & = 0 \quad \mbox{if } M \in \cN_K   \mbox{ and } M \neq N.
	\end{split}
\end{equation}
Thus $\mbox{span}\{\Psi_N\}$ is the kernel of the system \cref{local-construction}. However, it is obvious that 
$\Psi_N$ does not satisfy the constraint equation $\sum\limits_{F \in \cE_N} \mathfrak{s}_N(F) h_F  \mu|_{F}(N) =0$ for $N \in \cN_I$. 
Thus \cref{local-construction}  has a unique solution $\theta_N \in M_h$ under the constraint for $N \in \cN_I$.
%, that is a  unique solution in $M_h$.
%For $N \in \cN_{\partial}$, in two dimensions, note that there are $n_N+1$ edges and two of them are on the boundary. The solution is then unique since there are $n_N-1$ unknowns with $n_N-1$ independent equations. 
%This completes the proof of the lemma.
For $N \in \cN_{\partial}$, note that there are $n_N-1$ interior facets. The solution is then unique since there are $n_N-1$ unknowns with $n_N-1$ independent equations. 
This completes the proof of the lemma.
\end{proof}

\begin{lemma}
	Let $\theta_N$ be defined in \cref{local-construction} and $\theta_h$ be the solution of \cref{mixed-formulation}. Then we have that
	\begin{equation}\label{equivalence}
	\theta_h = \sum_{N \in \cN} \theta_N.
	\end{equation}
\end{lemma}
\begin{proof}
	From \cref{residual1}, it suffices to prove that
\begin{equation}\label{equivalence-a}
\sum_{N \in \cN}  b(\theta_N, \lambda_M \chi_K) = r(\lambda_M \chi_K), \quad \forall \,M \in \cN, \; \forall \,K \in \cT_M.
\end{equation}
By  \cref{local-construction}, we have that
\begin{equation}
	\begin{split}
	\sum_{N \in \cN}  b(\theta_N, \lambda_M \chi_K) = 
	\sum_{N \in \cN_K}  b(\theta_N, \lambda_M \chi_K)  =  b(\theta_M, \lambda_M \chi_K) = r(\lambda_M \chi_K).
	\end{split}
\end{equation}
This completes the proof of the lemma.
\end{proof}

\subsection{Computation of $\theta_N$ for $N \in \cN$}
%For the sake of simplicity, we present the explicit formulas in the two dimensions. 
%The generalization to the three dimensions can be achieved  in a similar fashion. 
For each $N \in \cN$, 
recall that we let $\{K_{i,N}\}_{i=1}^{n_N}$ be the clockwise oriented elements in $\cT_N$ and $n_N \ge 1$ is the number of elements in $\cT_N$. 
When $N \in \cN_I$, we let $F_i = K_{i-1,N} \cap K_{i,N}, i=1, \cdots, n_N (K_{0} = K_{n_N})$.  When $N \in \cN_\partial$, note that there are $n_N-1$ interior facets in $\cE_N$. %We rank $F_i$ counter-clock wisely.
Also let $M_{i}$ denotes the other vertices of $F_i$.
% When there is no risk of confusion, we omit the subscript $N$ for $K_{i,N}, F_{i,N}$ and $M_{i,N}$.
%Let $K_{i,N} \subset \cT_N$ be the element formed by vertices $N, M_{i}, M_{i+1}$, counterclockwisely. 

We firstly deal with the case $N \in \cN_I$.
From the second equation in \cref{local-construction}, it is easy to see that $\theta_N(M_i)=0$ for $i=1, \cdots, n_N$.
We then let
\[\theta_{i,N} := \theta_N|_{F_i}(N), \quad r_{i,N}:=r (\lambda_N \chi_{K_i}), \quad a_{i,N} := \mathfrak{s}_N (F_i)h_{F_i}. 
\]
From the first equation in \cref{local-construction}, a straight computation gives that
\begin{equation}\label{theta-n-1}
	a_{i,N} \theta_{i,N} - a_{i+1}\theta_{i+1,N} = 2 r_{i,N}, \quad i = 1, \cdots, n_N-1.
\end{equation}
Note that we only used $n_N-1$ equations since the last one is linearly dependent.
%When the index is $N+1$, its equivalent to $1$ (mod equivalent).
The constraint  provides the last equation if $N \in \cN_I$:
\begin{equation}\label{theta-n-2}
%	\sum_{i=1}^{n_N} \mathfrak{s}_N (F_i) sign(K_i,F_i)  a_{i,N} \theta_{i,N}=0 \Rightarrow 
	\sum_{i=1}^{n_N} a_{i,N} \theta_{i,N} =0.
\end{equation}
Combining \cref{theta-n-1} and \cref{theta-n-2} yields a $n_N \times n_N$ non-singular local system. It is helpful to denote the unknowns by $ \tilde{\theta}_{i,N} :=a_{i,N}  \theta_{i,N}$, such that the matrix of the previous local system has constant coefficients and depends only on $n_N$ .

If $N \in \cN_{\partial}$, note that $\theta_N|_{F}(N)= 0$ for $F\in \cE_{\partial}$, i.e., $\theta_{1,N}= \theta_{n_N,N} = 0$. Hence we can explicitly compute $\theta_{i,N}$ for $i =1, \cdots, n_N-1$. Thus no local $n_N \times  n_N$ problem is required to solve for boundary elements.

By a simple calculation, we also have the following estimate:
\begin{equation}\label{estimate-r-N-to-theta-N}
\sum_{i=1}^{n_N} \tilde{\theta}_{i,N}^2 \lesssim \sum_{i=1}^{n_N} r_{i,N}^2 \,\Longleftrightarrow\, \|\{\tilde{\theta}_{N}\}\|
	\lesssim \|\{r_{N}\}\|,
%	A_N x_N = r_N \Rightarrow x_N = A_N^{-1} r_N  \Rightarrow \|x_N\| \le  \|A_N^{-1}\| \|r_N\| \lesssim h_N^{-1} \|r_N\|.
\end{equation}
where $\{r_{N}\}$ and $\{\tilde{\theta}_{N}\}$ are the vectors in $\mathbb R^{n_N}$ formed by $(r_{i,N})_{i}$ and $(\tilde{\theta}_{i,N})_i$, respectively. Here  $\| \cdot\|$ denotes the Euclidean $2$-norm.
%Let $w = \lambda_N \chi_{K_i}$.

We also obtain the following estimate for $\{r_N\}$ which will be used later.
From the definition of $r_{i,N}$, integration by parts, \cref{eq:boundary_trace}, $\jump{u_h}|_{\cE_I}=0$ and  direct computations on norms of $\lambda_N \chi_{K_i}$, we have
\begin{equation}\label{estimate-r-N}
\begin{split}
	r_{i,N} =& r(\lambda_N \chi_{K_i})\\
		=&(f,\lambda_N \chi_{K_i})_{\Omega_h} - (\nabla u_h, \nabla (\lambda_N \chi_{K_i}))_{\Omega_h}
	- \left<g_h - u_h , \partial_{\bn_h} (\lambda_N \chi_{K_i})\right>_{\partial \Omega_h} \\
	&+\left<\lambda_N \chi_{K_i} , \partial_{\bn_h} u_h\right>_{\partial \Omega_h} 
	+ \left< \{ \partial_{\bn_F} u_h\}, \jump{\lambda_N \chi_{K_i}} \right>_{\cE_I \cap \Omega_h}
\\
&+ \sum_{K \in \cT_h^b} \beta h_K^{-1}\left<g_h -u_h,\lambda_N \chi_{K_i}\right>_{\Gamma_K}
-\gamma\sum_{F \in \cE_{g}}  h_F \left< \jump{\partial_{\bn_F} (\lambda_N \chi_{K_i})},\jump{\partial_{\bn_F} u_h}\right>_F\\
=&(f,\lambda_N \chi_{K_i})_{K_i \cap \Omega_h} 
	- \left<g_h - u_h , \partial_{\bn_h} (\lambda_N \chi_{K_i})\right>_{\Gamma_{K_i}} + {\beta}{h^{-1}_{K_i}}\left<g_h -u_h,\lambda_N \chi_{K_i}\right>_{\Gamma_{K_i}}\\
	&-\!\!\!\sum_{F \in \cE_I \cap \cE_{K_i}} \left<  \jump{ \partial_{\bn_F} u_h}, \{\lambda_N \chi_{K_i}\} \right>_{F \cap \Omega_h} 
	-\!\!\!
	\gamma\sum_{F \in \cE_{g} \cap \cE_{K_i}}  h_F \left< \jump{\partial_{\bn_F} (\lambda_N \chi_{K_i})},\jump{\partial_{\bn_F} u_h}\right>_F\\
%	 =& (\nabla u_h, \nabla (\lambda_N \chi_{K_i}) )_{K_i} 
%%	-\sum_{F \in \cE_K \cap \cE_I} \left< \{ \partial_n u_h\}, \jump{\lambda_N \chi_{K_i}} \right>_{F} - (f,  \lambda_N \chi_{K_i})_{K_i}\\
%	 =& \sum_{F \in \cE_K} \left<\partial_{n_F} u_h, \jump{ \lambda_N \chi_{K_i}} \right>_{F} 
%	-\sum_{F \in \cE_K \cap \cE_I} \left< \{ \partial_n u_h\}, \jump{\lambda_N \chi_{K_i}} \right>_{F} - (f,  \lambda_N \chi_{K_i})_{K_i}\\
%	\lesssim & \sum_{F \in \cE_K \cap \cE_I \cap \cE_N} h_F^{1/2}\| \jump{ \partial_n u_h}\|_F+ h_K \|f\|_{K_i}\\
	\lesssim & h_{K_i} \|f \|_{K_i \cap \Omega_h}  
	+h_{K_i}^{-1/2} \|g_h - u_h\|_{\Gamma_{K_i}}  
	+\sum_{F \in \cE_I \cap \cE_{K_i}}h_F^{1/2} \| \jump{ \partial_{\bn_F} u_h}\|_{F}.
\end{split}
\end{equation}
Immediately, we also have for any node $N$ that
\begin{equation}\label{residual-estimate}
\begin{split}
	\|\{\tilde{\theta}_{N}\}\|
	&\lesssim  \|\{r_{N}\}\|\\
	&\lesssim  \sum_{F \in \cap \cE_N \cap \cE_I} h_F^{1/2}\| \jump{ \partial_{\bn_F} u_h}\|_F+ 
	\sum_{K \in \cT_N} \left( h_K \|f\|_{K\cap \Omega_h} +h_K^{-1/2} \|g_h - u_h\|_{\Gamma_K}  \right).
	\end{split}
\end{equation}
%Note that the estimates from \cref{estimate-r-N-to-theta-N} to \cref{residual-estimate} are also applicable to 3D.
%\begin{equation}\label{residual-estimate}
%\begin{split}
%	\|\{r_{i,N}\}\|
%	\lesssim & \sum_{F \in \cap \cE_N \cap \cE_I} h_F^{1/2}\| \jump{ \partial_n u_h}\|_F+ 
%	\sum_{K \in \cT_N} \left( h_K \|f\|_{K\cap \Omega_h} +h_K^{-1/2} \|g_h - u_h\|_{\Gamma_K}  \right),\\
%	\|\{\theta_{i,N}\}\|
%	\lesssim & \sum_{F \in \cap \cE_N \cap \cE_I} h_F^{-1/2}\| \jump{ \partial_n u_h}\|_F+ 
%	\sum_{K \in \cT_N} \left( \|f\|_{K\cap \Omega_h} +h_K^{-3/2} \|g_h - u_h\|_{\Gamma_K}  \right).
%\end{split}
%\end{equation}

%The right hand side values $r_{i,N} $ can be obtained from \cref{sec-r-h}. 

\subsection{Flux reconstruction}
In this subsection, we locally recover a flux for each element $K \in \cT_h$. The element-wise construction is completely explicit and based on the computation of $u_h$ and $\theta_h$. 
%Firstly we consider the case when $K \subset \Omega$, i.e., $K$ is completely inside the domain.
%We define $\bsigma_K \in RT^1(K)$ such that
%\begin{equation}\label{flux-construction-1}
%	\begin{split}
%		(\bsigma_K, \bftau)_K &= (\nabla u_h, \bftau) + 
%		\gamma\sum_{F \in \cE_{g} \cap \cE_K }  
%		h_F \left< \jump{\partial_{\bn} u_h},\jump{\bftau} \cdot \bn_F\right>_F \;
%		\forall \bftau_0 \in P^0(K)^2\\
%		\left<  \bsigma \cdot \bn_F, \mu \right> & =\left< \{ \partial_{\bn_F} u_h\}, \mu \right>_F + b_F(\theta_h, \mu)
%		\; \forall \mu \in P^1(F) 
%	\end{split}
%\end{equation}
%where $b_F(\theta_h, w) = h_F \sum_{Z \in \cN_F} \dfrac{1}{2} \theta_h (Z) w(Z)$.
%From \cref{flux-construction-1}, \cref{residual1} and \cref{residual2} it is easy to check that
%\[
%	(\nabla \cdot \bsigma, w)_K = (f, w)_K \quad \forall w \in P^1(K).
%\]
%
%Next we consider the element $K \in \cT_h^{b}$, i.e., $K \not\subset \Omega$.
Denote the  $H(\mbox{div}; \triangle_h)$ conforming Raviart-Thomas ($RT$) space of index $1$ with respect to $\cT_h$
by
\[
RT_h= \left\{ \bftau \in H(\mbox{div}; \triangle_h) \,: \, \bftau|_K \in RT^{1}(K),  \;\forall\, K\in\cT_h \right\},
\]
where $RT^{1}(K) = \mathbb{P}_{1}(K)^d + \bfx \, \mathbb{P}_{1}(K) $.
%Let
%\[
% \Sigma_f^{k}(\cT) = \left\{ \bftau \in RT: \nabla \cdot \bftau =f_{1} \,\mbox{ in } \, \triangle_h \right\}.
%\]   
%where $f_1$ 
On a triangular  element $K \in \cT_h$, a vector-valued function $\bftau$ in $RT^1(K)$  is characterized 
by the following degrees of freedom (see Proposition 2.3.4 in \cite{boffi2013mixed}):
 \[
 	\int_K \bftau \cdot \bfzeta \,dx , \quad \forall \, \bfzeta \in \mathbb{P}_{0}(K)^d,	
 \]
 and
 \[	
 	\int_F ( \bftau \cdot \bn_F) \,w \,ds, \quad \forall \,w \in \mathbb{P}_{1}(F) \mbox{ and } \; \forall \, F \in \cE_K.
 \]

For each element $K \in \cT_h$, we define $\bsigma_K \in RT^1(K)$ such that for all $ \bfzeta \in \mathbb{P}_{0}(K)^d$ and  for all $ F \in \cE_K$ and $w \in \mathbb{P}_{1}(F)$ it satisfies:
\begin{equation}\label{flux-construction}
	\begin{split}
		&(\bsigma_K, \bfzeta)_K =
		(\nabla u_h, \bfzeta)_K + \gamma\sum_{F \in \cE_{g} \cap \cE_K }  
		h_F \left< \jump{\partial_{\bn_F} u_h},\jump{\bfzeta \cdot \bn_F}\right>_F 
		+ \left<g_h - u_h, \bfzeta \cdot \bn_h \right>_{\Gamma_K},\\
		 &\left<  \bsigma_K \cdot \bn_F, w \right>_F  =
		\left< \{ \partial_{\bn_F} u_h\}, w \right>_F 
		 - b_F(\theta_h, w) \quad \mbox{if } F \in \cE_I,\\
		 &\left<  \bsigma_K \cdot \bn_F, w \right>_F  =
		\left<  \partial_{\bn_F} u_h, w \right>_F 
		  \quad \mbox{if } F \in \cE_{\partial}\setminus \Gamma_K,\\
		&\left<  \bsigma_K \cdot \bn_F, w \right>_F  =
		\left< \partial_{\bn_F} u_h, w \right>_F + 
		\dfrac{\beta}{h_K}\left<g_h -u_h, w \right>_{\Gamma_K} \quad \mbox{if } F = \Gamma_K,
	\end{split}
\end{equation}
%where  $w_{ K}$ is the extension of $w$ to the subspace of $P^1(K)$ which is spanned by the nodal basis functions with nodes on $F$, and 
where 
$$b_F(\theta_h, w) = \frac {h_F}{2} \sum_{N\in\cN\cap F} \theta_{h|F} (N) w(N).$$ 
\begin{remark}
One can also reconstruct the flux in the Raviart-Thomas space of index $0$ similarly to \cref{flux-construction}. Note that there are now no interior degrees of freedom thus the first equation in \cref{flux-construction} is not needed, whereas the edge degrees of freedom are tested with $w \in \mathbb{P}_{0}(F)$.
\end{remark}

\begin{remark}
	Note that if $F=\Gamma_K \in \cE_K$, we then have that
	\[
		 \bsigma_K \cdot \bn_F =  \partial_{\bn_F} u_h+ \dfrac{\beta}{h_K}(g_h -u_h).
	\]
\end{remark}
%For elements $K$ such that there exists one $F \in \cE_K$ such that $F \in \Omega_h^c$, we define 
%\begin{equation}\label{flux-construction-a}
%	\begin{split}
%		&(\bsigma_K, \bfzeta)_K =
%		(\nabla u_h, \bfzeta)_K + \gamma\sum_{F \in \cE_{g} \cap \cE_K }  
%		h_F \left< \jump{\partial_{\bn_F} u_h},\jump{\bfzeta \cdot \bn_F}\right>_F 
%		+ \left<g_h - u_h, \bfzeta \cdot \bn_h \right>_{\Gamma_K},\\
%		&\left<  \bsigma_K \cdot \bn_F, w \right>_F  =
%		\left< \{ \partial_{\bn_F} u_h\}, w \right>_F + 
%		\textcolor{red}{\dfrac{\beta}{h_K}\left<g_h -u_h, w \right>_{\Gamma_K}}
%		 - b_F(\theta_h, w).
%	\end{split}
%\end{equation}
%On the boundary edge $F \in \cE_{\partial}$, we define that $\{u\}|_F = u$, and if, furthermore, $F \subset \Omega^c$ we also define  $b_F(\theta_h, \mu)=0$ since $\theta_h$ is not defined on $F$.
We then define the global recovered flux by
\begin{equation}\label{flux-construction-1}
	\bsigma_h = \sum_{K \in \cT_h} \bsigma_K.
\end{equation}

%\textcolor{red}{Programming technique: generate two functions $\bsigma_0 \in DG^0(\cT_h)^d$ and $\bsigma_0 \in DG^0(\cT_1)^d$ in FENICS such that
%for all the DOFs in the interior of elements there holds:
%\[
%(\bsigma_0, \bftau)_K = (\bsigma_K, \bftau)_K \quad \forall \bftau \in P^0(K)^d
%\]
%and for all the DOFs on the boundary of elements there holds
%\[
%	\left<  \bsigma_1 \cdot \bn_F, \mu \right>_F  = \left<  \bsigma_K \cdot \bn_F, \mu \right>_F.
%\]
%}
%\[I_h(\bsigma_0) = I_h(\bsigma_h)\]

Recall that on a cut element $K \in \cT_h^b$, $f$ is only defined on $K \cap \Omega_h$. We next introduce an extension of $f$ to the whole cut element $K$. For each $K \in \cT_h^b$ and $K \not \subset \Omega_h$, 
%we label it type I  if $F \cap \Omega_h \neq \emptyset $ for all $F \in \cE_K$. Note that this implies $\cE_K\subset\cE_I$. Otherwise, the cut element is labelled with type II.For each type I cut element $K $, 
we extend $f$ to $K \cap \Og_h^c$ such that
$f|_{K \cap \Og_h^c}$ is linear and also satisfies that for all $w \in \mathbb{P}_1(K)$,
\begin{equation}\label{tilde-f}
	\begin{split}
		( f, w)_{K \cap \Omega_h^c} =\beta h_K^{-1}\left< g_h - u_h, w\right>_{\Gamma_K}
		+\sum_{F \in \cE_K \cap \cE_I}  \frac{1}{2}\left<\mathfrak{s}_K(F) \jump{\partial_{\bn_F} u_h}, \jump{w} \right>_{F \cap \Omega_h^c}.
	\end{split}
\end{equation}
%For a type II element with $ \hat F \in \cE_K \cap \cE_{\partial}$, we define
%\begin{equation}\label{tilde-f-ii}
%	\begin{split}
%		( f, w)_{K \cap \Omega_h^c} =\beta h_K^{-1}\left< (g_h - u_h), w -w_{\hat F}\right>_{\Gamma_K}
%		+\sum_{F \in \cE_I \cap \cE_K }\frac{1}{2}\left<\mathfrak{s}_K(F) \jump{\partial_{\bn_F} u_h}, \jump{w} \right>_{F \cap \Omega_h^c},
%	\end{split}
%\end{equation}
%where $w_{\hat F}$ is the projection of $w\in P^1(K)$ onto the subspace $P_{\hat F}(K)$. 
Note that if the right-hand side is $0$, this represents a zero extension.

Let $\Pi_1$ be the $L^2$ projection operator onto the space $DG_h$ defined on $\cT_h$.
\begin{lemma}
	Let $\bsigma_h$ be defined in \cref{flux-construction-1}. Then we have that $\bsigma_h \in RT_h$ and
	\begin{equation}\label{local-conservation}
	\begin{split}
	-\divvr \bsigma_h = \Pi_1(f) \quad \forall \, K \in \cT_h. %\quad
%	\mbox{and} \quad
%	-\divvr \bsigma =\Pi(f) \quad \forall \, K \in \cT_h^{b}.
	\end{split}
	\end{equation}
\end{lemma}

\begin{proof}
	By its definition, it is easy to see that $\bsigma_h \in RT_h$.
	Firstly, we note that for all $w \in DG_h$ we have
\begin{equation}\label{residual2-a}
	\begin{split}
	b(\theta_h,w) =& l_h(w) - \tilde a_h(u_h, w)\\
	=&
	(f,w)_{\Omega_h} - (\nabla u_h, \nabla w)_{\Omega_h}
	- \left<g_h - u_h , \partial_{\bn_h} w\right>_{\partial \Omega_h} 
	+\left<w , \partial_{\bn_h} u_h\right>_{\partial \Omega_h}\\
&+ \sum_{K \in \cT_h^b} \dfrac{\beta}{h_K}\left<g_h -u_h,w\right>_{\Gamma_K}
-\gamma\sum_{F \in \cE_{g}}  h_F \left< \jump{\partial_{\bn_F} w},\jump{\partial_{\bn_F} u_h}\right>_F\\
&+ \left< \{ \partial_{\bn_F} u_h\}, \jump{w} \right>_{\cE_I \cap \Omega_h}
	\end{split}
\end{equation}
since $\jump{u_h}=0$.	To prove \cref{local-conservation}, we  first consider the case of  $K \in \cT_h^{int}$, i.e., $K \subset \Og_h$, $K \cap \partial \Og_h = \emptyset$.
	Let $w \in DG_h$ such that $w|_K\in \mathbb{P}_1(K)$ and $w$ vanishes elsewhere. From \cref{residual2-a} we have
	\begin{equation}\label{residual3}
	\begin{split}
b(\theta_h, w) =&
	(f,w)_K - (\nabla u_h, \nabla w)_K
-\gamma\sum_{F \in \cE_{g} \cap \cE_K}  h_F \left< \jump{\partial_{\bn_F} u_h},\jump{\partial_{\bn_F} w}\right>_F\\
&+ \left< \{ \partial_{\bn_F} u_h\}, \jump{w} \right>_{\cE_K}.
	\end{split}
\end{equation}
By integration by parts, \cref{flux-construction}, the fact that $b(\theta_h, w) = \sum\limits_{F \in \cE_K} b_F(\theta_h, \jump{w})$, and \cref{residual3}, we have
\begin{equation}\label{residual4}
	\begin{split}
	(\nabla \cdot \bsigma_h, w)_K =& -(\bsigma_h, \nabla w)_K + \left< \bsigma_h \cdot \bn_K, w \right>_{\cE_K}\\
	=& -(\bsigma_h, \nabla w)_K + \left< \bsigma_h \cdot \bn_F, \jump{w} \right>_{\cE_K}\\
	=&-(\nabla u_h, \nabla w)_K-\gamma\sum_{F \in \cE_{g} \cap \cE_K}  h_F \left< \jump{\partial_{\bn_F} u_h},\jump{\partial_{\bn_F} w}\right>_F\\
	&+\left< \{ \partial_{\bn_F} u_h\}, \jump{w} \right>_{\cE_K} - b(\theta_h, w) = -(f,w)_K,
	\end{split}
\end{equation}
which yields \cref{local-conservation} for all interior elements. 

Now consider the second case when $K\in \cT_h^b$. Again, let  $w \in DG_h$ such that $w|_K\in \mathbb{P}_1(K)$ and $w$ vanishes elsewhere.
From \cref{residual2-a} we now have
	\begin{equation}
	\begin{split}
b(\theta_h, w) =&
	(f,w)_{K \cap \Omega_h} - (\nabla u_h, \nabla w)_{K \cap \Omega_h}
	- \left<g_h - u_h , \partial_{\bn_h} w\right>_{\Gamma_K} 
	+\left<w , \partial_{\bn_h} u_h\right>_{\Gamma_K}\\
&+  \beta h_K^{-1}\left<g_h -u_h,w\right>_{\Gamma_K}
-\gamma\sum_{F \in \cE_{g} \cap \cE_K}  h_F \left< \jump{\partial_{\bn_F} w},\jump{\partial_{\bn_F} u_h}\right>_F\\
&+ \left< \{ \partial_{\bn_F} u_h\}, \jump{w} \right>_{\cE_I \cap \cE_K \cap \Omega_h}.
	\end{split}
\end{equation}
Applying integration by parts on $K \cap \Omega_h^c$ gives
\begin{equation}
	\begin{split}
		\left<w , \partial_{\bn_h} u_h\right>_{\Gamma_K} &= -(\nabla u_h, \nabla w)_{K \cap \Omega_h^c}
		+ \left< \partial_{\bn_K} u_h, w \right>_{\partial (K \cap \Omega_h^c) \setminus \Gamma_K}\\
		&= -(\nabla u_h, \nabla w)_{K \cap \Omega_h^c}
		+ \left< \partial_{\bn_F} u_h, \jump{w} \right>_{\cE_I \cap \cE_K \cap \Omega_h^c}
		+ \left< \partial_{\bn_K} u_h, w \right>_{\cE_K \cap \cE_{\partial}}
	\end{split}
\end{equation}
which, combining with the following equation,
\[
	\left< \partial_{\bn_F} u_h, \jump{w} \right>_{\cE_I \cap \cE_K \cap \Omega_h^c} = 
	\left< \{\partial_{\bn_F} u_h\}, \jump{w} \right>_{\cE_I \cap \cE_K \cap \Omega_h^c}
	+\dfrac{1}{2} \left< \mathfrak{s}_K(F)\jump{\partial_{\bn_F} u_h}, \jump{w} \right>_{\cE_I \cap \cE_K  \cap \Omega_h^c},
\]
implies
\begin{equation}\label{residual5}
	\begin{split}
		\left<w , \partial_{\bn_h} u_h\right>_{\Gamma_K} =& -(\nabla u_h, \nabla w)_{K \cap \Omega_h^c}
		+ \left< \{\partial_{\bn_F} u_h\}, \jump{w} \right>_{\cE_I \cap \cE_K \cap \Omega_h^c}\\
	&+\dfrac{1}{2}\left< \mathfrak{s}_K(F)\jump{\partial_{\bn_F} u_h}, \jump{w} \right>_{\cE_I \cap \cE_K \cap \Omega_h^c}
		+ \left< \partial_{\bn_K} u_h, w \right>_{\cE_K \cap \cE_{\partial}}.
	\end{split}
\end{equation}
Note that the previous relation also holds in the case $\Gamma_K=\cE_K \cap \cE_{\partial}$. %when it simply reads $\left<w , \partial_{\bn_h} u_h\right>_{\Gamma_K}=\left< \partial_{n_F} u_h, \jump{w} \right>_{\cE_K \cap \cE_{\partial}}$.

%Assume that $K$ is a type I cut element. 
Combining all above with the definitions in \cref{tilde-f} and \cref{flux-construction} gives
	\begin{equation}\label{residual6}
	\begin{split}
b(\theta_h, w)=&
	(f,w)_{K} - (\nabla u_h, \nabla w)_{K }
	- \left<g_h - u_h , \partial_{\bn_h} w\right>_{\Gamma_K} \\
&
-\gamma\sum_{F \in \cE_{g} \cap \cE_K}  h_F \left< \jump{\partial_{\bn_F} w},\jump{\partial_{\bn_F} u_h}\right>_F
+ \left< \{ \partial_{\bn_F} u_h\}, \jump{w} \right>_{ \cE_K}\\
=&
	(f,w)_{K} 
	- (\bsigma_h, \nabla w)_K
+ \left< \{ \partial_{\bn_F} u_h\}, \jump{w} \right>_{ \cE_K}.
	\end{split}
\end{equation}
Again, we have $b(\theta_h, w) = \sum\limits_{F \in \cE_K} b_F(\theta_h, \jump{w})$. By using \cref{flux-construction}, the previous equality gives  
\begin{equation}\label{residual7}
	0=(f,w)_{K} - (\bsigma_h, \nabla w)_K+ \left< \bsigma_h \cdot \bn_F, \jump{w} \right>_{\cE_K}
\end{equation}
so we obtain \cref{local-conservation} thanks to the integration by parts formula.
%Finally, assume that $K$ is a type II cut element and let $ \hat F = \cE_K \cap \cE_{\partial}$. Then, similarly to the previous case, we get by using \cref{tilde-f-ii} and \cref{flux-construction} that
%\begin{equation}\label{residual8}
%\begin{split}
%b(\theta_h, w)=&
%	(f,w)_{K} 
%	- (\bsigma_h, \nabla w)_K
%+ \left< \{ \partial_{\bn_F} u_h\}, \jump{w} \right>_{ \cE_I\cap \cE_K}\\
%      &+\left< \partial_{\bn_K} u_h, w \right>_{\hat F} + 
%		\dfrac{\beta}{h_K}\left<g_h -u_h, w_{ \hat F} \right>_{\Gamma_K}.
%		\end{split}
%\end{equation}
%Now, we have $b(\theta_h, w) = \sum\limits_{F \in \cE_K\cap \cE_I} b_F(\theta_h, \jump{w})$. By using \cref{flux-construction}, the fact that $\bn_{\hat F}=\bn_K$ and that the extension of $w|_{\hat F} $ to $ P_{\hat F}(K) $ coincides with the projection $w_{\hat F}$ of $w$ onto $ P_{\hat F}(K) $, we obtain \cref{residual7}. 
This completes the proof of the lemma.
\end{proof}

\begin{remark}
For the cut elements, it is not obvious to construct a flux that is both locally conservative in the cut part $K \cap \Og_h$ and, at the same time, maintains continuous normal flux. The technique of applying integration by parts in \cref{residual5} renders the problem to be a more regular problem and completes partial elements to full elements.
\end{remark}

\section{Application in the a posteriori error estimation}\label{posteriori}
For the sake of simplicity, we assume in this section that $f|_K\in\mathbb P_1(K)$ on any $K\in \cT_h \setminus \cT_h^b$. 
In the adaptive procedure, we define the following local error indicators 
\begin{equation} \label{indicator-a}
	\eta_{K,1} = \| \bsigma_h - \nabla u_h\|_K, \quad \eta_{K,2} = \| \bsigma_h - \nabla u_h\|_{K \cap \Omega_h}  \quad \forall K \in \cT_h
\end{equation}
and the corresponding estimators:
\begin{equation} \label{estimator-a}
	\eta_1 = \sqrt{\sum_{K \in \cT_h}\eta_{K,1}^2}, \quad \eta_2 = \sqrt{\sum_{K \in \cT_h}\eta_{K,2}^2} \,.
\end{equation}

\subsection{Reliability}

Let $\tilde e \in H^1(\Omega)$ be the lifting such that $\tilde e = e := u- u_h$ on $\partial \Omega$ and
\[
	\|\tilde e\|_{H^1(\Omega)} = \| e\|_{H^{1/2}(\partial \Omega)}.
\]
\begin{theorem}[Reliability]
Let $\bsigma_h$ be given by (\ref{flux-construction-1}) and $u_h$ be the CutFEM solution in (\ref{eq:fem}).
We have the following reliability result:
\begin{equation}\label{reliability-1}
	\begin{split}
	\| \nabla (u - u _h)\|_{\Omega} \le  \| \bsigma_h  - \nabla u_h\|_{\Omega} + 2 \| \nabla \tilde e\|_{\Og} + C \epsilon,
	\end{split}
\end{equation}
where the  constant $C$ is independent of the mesh size and mesh-domain intersection,
and
\[
  \epsilon = \sqrt{\sum_{K \in \cT_h^b} h_K^2 \|f  - \Pi_1(f) \|_{K \cap \Omega}^2}.
\]

\end{theorem}

\begin{remark}
Thanks to the assumption that $\Omega \subset \Omega_h$, we have that
\[
	\| \bsigma_h  - \nabla u_h\|_{\Omega} \le \| \bsigma_h  - \nabla u_h\|_{\Omega_h} \le \| \bsigma_h  - \nabla u_h\|_{\triangle_h}.
\]
Therefore both $\eta_1$ and $\eta_2$ could serve as the error estimator in the AMR procedure.
\end{remark}

\begin{proof}
By triangle inequality, we firstly have the following bound:
\begin{equation}\label{reliability-a}
	\begin{split}
		\|\nabla e\|_{\Omega} &\le \|\nabla (e - \tilde e)\|_{ \Omega}  +
		\|\nabla \tilde e\|_{\Omega} \\
		&= \sup_{v \in H_0^1(\Omega)} \dfrac{(\nabla (e - \tilde e), \nabla v)}{\|\nabla v\|_{\Omega}} 
		+ \|\nabla \tilde e\|_{ \Omega}\\
		& \le 
		 \sup_{v \in H_0^1(\Omega)} \dfrac{(\nabla e, \nabla v)}{\|\nabla v\|_{\Omega}} 
		+2 \|\nabla \tilde e\|_{ \Omega}.
	\end{split}
\end{equation}
To bound $(\nabla e, \nabla v)_{\Omega} $ we have
\begin{equation}
	\begin{split}
	|(\nabla e, \nabla v)_{\Omega} |&= |(\nabla u - \bsigma_h, \nabla v)_{\Omega}+
	(\bsigma_h - \nabla u_h, \nabla v)_{\Omega}|\\
	&\le \sum_{K \in \cT_h^b}|(f - \Pi_1(f), v)_{K \cap \Omega} |+ |(\bsigma_h - \nabla u_h, \nabla v)_{\Omega}|\\
%	&\le \sum_{K \in \cT_h^b} \|P^1(f) - P^1(f)\|_{K \cap \Omega} \|v\|_{K \cap \Omega}+
%	\|f -P^1(f) \|_{K \cap \Omega} \|v\|_{K \cap \Omega}
%	 + (\bsigma_h - \nabla u_h, \nabla v)_{\Omega}\\
	&\lesssim
	\sum_{K \in \cT_h^b} 
	\|f -\Pi_1(f) \|_{K \cap \Omega} \|v\|_{K }
	+ \|\bsigma_h - \nabla u_h\|_{\Omega} \|\nabla v\|_{\Omega}\\
%	&\le 
%	\sum_{K \in \cT_h^b} C h_K\left(\|f\|_{K \cap \Omega \cap \Omega_h^c}  + \|f -P^1(f) \|_{K \cap \Omega}\right) \|\nabla v\|_{K}
%	+ \|\bsigma_h - \nabla u_h\|_{\Omega} \|\nabla v\|_{\Omega}\\
	&\le
	C \epsilon   \|\nabla v\|_{\Og}
	+ \|\bsigma_h - \nabla u_h\|_{\Omega} \|\nabla v\|_{\Omega}
	\end{split}
\end{equation}
where we used (\ref{poincare}) for the last inequality. This completes the proof of the theorem.
\end{proof}

\begin{remark} It is useful to note that for any $K\in \cT_h^b$, we have that
\begin{equation}\label{reliability-d_0}
 \|\Pi_1(f)\|_{K}\lesssim |K|^{-1/2}\int_K|f|.
\end{equation}
Indeed, denoting by $(a_i)_{1\leq i\leq d+1}$ the values taken by $\Pi_1(f)$ at the vertices of $K$ and by $\lambda_i$ the corresponding nodal basis functions on $K$, we have using that $0\leq\lambda_i\leq 1$,
\begin{equation*}
\begin{split}
|K|\sum_{i=1}^{d+1}a_i^2\simeq \|\Pi_1(f)\|_{K}^2&=\int_K f \bigg(\sum_{i=1}^{d+1}a_i \lambda_i\bigg)\\
&\lesssim \int_K|f| \bigg(\sum_{i=1}^{d+1}|a_i|\bigg)\lesssim |K|^{-1/2} \|\Pi_1(f)\|_{K} \int_K|f|,
\end{split}
\end{equation*}
which yields the desired estimate.
\end{remark}

\begin{remark}
By the definition of $f$ in \cref{tilde-f}, $f|_{\Og_h^c}$ can be bounded as follows using Cauchy-Schwarz inequality:
\begin{equation*}
	\begin{split}
		&|(f,w)_{K \cap \Og_h^c}|
		\lesssim
		\beta h_K^{-1}  \|g_h - u_h\|_{\Gamma_K} \| w\|_{\Gamma_K} + \sum_{F\in \cE_I \cap \cE_K} \|\jump{\partial_{\bn_F} u_h}\|_{F\cap \Omega_h^c}  \| w\|_{F \cap \Og_h^c}\\
		\lesssim&
		\left(\beta h_K^{-1} \frac{ |\Gamma_K|^{1/2} }{|K \cap \Og_h^c|^{1/2}}\|g_h - u_h\|_{\Gamma_K}  + \!\! \sum_{F\in \cE_I \cap \cE_K} \frac{|F\cap \Omega_h^c|^{1/2}}{|K \cap \Og_h^c|^{1/2}} \|\jump{\partial_{\bn_F} u_h}\|_{F\cap \Omega_h^c} \right) \| w\|_{K \cap \Og_h^c},
	\end{split}
\end{equation*}
for any $w\in \mathbb P_1(K \cap \Og_h^c)$.
Therefore, we have
\begin{equation*}\label{reliability-b_0}
	\begin{split}
	|K \cap \Og_h^c|^{1/2}\|f\|_{K \cap \Og_h^c} \lesssim 	
	\beta h_K^{-1}  |\Gamma_K|^{1/2}\|g_h - u_h\|_{\Gamma_K}  + \!\!\!\!\!\!\!\sum_{F\in \cE_I \cap \cE_K} |F\cap \Omega_h^c|^{1/2} \|\jump{\partial_{\bn_F} u_h}\|_{F \cap \Og^c_h}.
	\end{split}
\end{equation*}
Hence, using that
\begin{equation*}\label{reliability-c_0}
\|f - \Pi_1(f)\|_{K \cap \Og} \le \|f - \Pi_1(f)\|_{K \cap \Og_h} \le \|f\|_{K \cap \Og_h} + \|\Pi_1(f)\|_{K}
\end{equation*}
as well as \cref{reliability-d_0} and
\begin{equation*}\label{reliability-c_00}
 \|\Pi_1(f)\|_{K}\lesssim \frac{1}{|K|^{1/2}}\int_K|f|\lesssim \frac{|K\cap \Og_h|^{1/2}}{|K|^{1/2}} \|f\|_{K \cap \Og_h} +\frac{|K\cap \Og_h^c|^{1/2}}{|K|^{1/2}} \|f\|_{K \cap \Og_h^c},
\end{equation*}
we obtain the next bound for  $\epsilon$:
%\[
%\epsilon^2 \lesssim  \sum_{K \in \cT_h^b} 
%	\left( h_K^2\|f\|_{K \cap \Og_h}^2 + \frac{|\Gamma_K|}{|K|}\|g_h - u_h\|_{\Gamma_K} ^2 + h_K^2 \sum_{F\in \cE_I \cap \cE_K} \frac{|F\cap \Omega_h^c|}{|K|}\|\jump{\partial_{\bn_F} u_h}\|_{F \cap \Og^c_h}^2\right).
%\]
%This can be further bounded as follows:
\begin{equation}\label{reliability-c_final}
\epsilon^2 \lesssim  \sum_{K \in \cT_h^b} 
	\left( h_K^2\|f\|_{K \cap \Og_h}^2 + h_K^{-2} |\Gamma_K|\|g_h - u_h\|_{\Gamma_K} ^2 +  \sum_{F\in \cE_I \cap \cE_K} |F\cap \Omega_h^c| \|\jump{\partial_{\bn_F} u_h}\|_{F \cap \Og^c_h}^2\right).
\end{equation}
\end{remark}

\begin{remark}
From the above estimate, we observe that $\epsilon$ can be bounded by the classical residual based error estimator. Moreover, the constants are uniformly bounded and independent of the domain-mesh intersection. 
We further note that $\epsilon$ can also be bounded by $\eta$ with an additional higher order oscillation term, 
\[
\mbox{osc}_1 =\sqrt{ h_K^2\|f - \Pi_{0,K \cap \Og_h} f\|_{K \cap \Og_h}^2 
+ h_K^{-d} |\Gamma_K|\|g_h - u_h - \Pi_{0, \Gamma_K} (g_h - u_h)\|_{\Gamma_K} ^2}.
\] Indeed, we could prove by equivalence of norms on finite dimensional space. Firstly note that $\eta=0$ implies $\bsigma_h = \nabla u_h$. Therefore, based on the properties of $\bsigma_h$, we have that $\jump{\nabla u_h \cdot \bn_F}=0$ for all $F \in \cE_I$ and that $f \equiv 0$ in $\cT_h \setminus \cT_h^b$. Moreover, from the first equation in \cref{flux-construction}, we immediately have $\Pi_{0,\Gamma_K}(g_h - u_h) = 0$ on $\Gamma_K$  for $K \in \cT_h^b$. From the second equation in \cref{flux-construction}, we have $\theta_h=0$ and thus, from \cref{residual1},  $\{r_{N}\}=0$ for each $N \in \cN_I$. Finally, the second equation in \cref{estimate-r-N} implies $ \Pi_{0,K \cap \Og_h} f=0$ on $K \in \cT_h^b$. Eventually, we have that $\eta + \mbox{osc}_1 =0$ implies $\epsilon=0$, and, therefore, $\epsilon \lesssim \eta+ \mbox{osc}_1$.
In the numerical computation, we discard the term $2 \| \nabla \tilde e\|_{\Og} + C \epsilon$. 
 $\|\nabla \tilde e\|_{\Og}$ is the so-called boundary correction error which we have thoroughly discussed in \cite{EburmanChe-2020}. It was shown that adding such error does not affect the overall convergence rate as well as the final meshes when the mesh is fine enough.
%	One could replace 
%	$\sum_{K \in \cT_h^b}\rho_K^{-1} \| g_h - u_h\|_{\Gamma_K}^2$ by $\|\nabla \tilde e\|_{\Omega}$ in the estimator. It will likely yield a better reliability constant. However, the computation will potentially be more costly.
\end{remark}
%\begin{remark}
%In the numerical computation, we discard the term $2 \| \nabla \tilde e\|_{\Og} + C \epsilon$. 
% $\|\nabla \tilde e\|_{\Og}$ is the so called boundary correction error which we have thoroughly discussed in \cite{EburmanChe-2020}. It was shown that adding such error does not affect the overall convergence rate as well as the final meshes when the mesh is fine enough.
%\end{remark}

\subsection{Efficiency}

\begin{lemma}\label{lem:effi-for-normal-elements}
Let $K$ be a given element in $ \cT_h$.
Then the following local efficiency result holds:
\begin{equation}\label{efficiency-b}
\begin{split}
&\| \bsigma_h -\nabla u_h\|_{K \cap \Omega_h}  \le \| \bsigma_h -\nabla u_h\|_K \lesssim \| \nabla (u-u_h)\|_{\tilde\Delta_K} \\
%+ \|h_K (f - \Pi_1(f))\|_{\tilde\Delta_K}  \\
&+\sum_{N \in \cN_K}\bigg(\sum_{F \in  \cE_N \cap \cE_g} h_F^{1/2}\| \jump{ \partial_{\bn_F} u_h}\|_F+ \!\!\!\! \sum_{K' \in \cT_N \cap \cT_h^b}h_{K'} \|f\|_{K'\cap \Og_h} +  h_{K'}^{-1/2} \|g_h - u_h\|_{\Gamma_{K'}}\bigg),
\end{split}
\end{equation}
where $\tilde \Delta_K$
is a local neighborhood of $K$ that does not contain elements in $\cT_h^b$ and the efficiency constant does not depend on the mesh size nor the domain-mesh intersection.
\end{lemma}
\begin{proof}
By the degrees of freedom for $RT^1(K)$ space, we have the following bound:
\begin{equation}\label{efficiency-b-a}
\|   \bsigma_h-\nabla u_h\|_K \lesssim
	\sum_{F \in \cE_K}h_F^{1/2} \|(  \bsigma_h-\nabla u_h) \cdot \bn_F\|_F
	+ \| \Pi_0( \bsigma_h-\nabla u_h) \|_K,
\end{equation}
where $\Pi_0$ is the $L^2$ projection onto the piecewise constant space on $\cT_h$. Let $F\in \cE_K$ and $p\in \mathbb P_1(F)$ arbitrary, then from \cref{flux-construction} we have:

\begin{equation}\label{efficiency-c}
\begin{split}
	\left< (  \bsigma_h -\nabla u_h) \cdot \bn_F, p \right>_F =&
	\left<   \{ \partial_{\bn_F} u_h\} -\nabla u_h \cdot \bn_F, p \right>_F - b_F(\theta_h,p)\\
	\le& \|   \jump {\partial_{\bn_F} u_h}\|_F \|p\|_F- b_F(\theta_h,p) \quad \mbox{if } F \in \cE_I,\\
	\left< (  \bsigma_h -\nabla u_h) \cdot \bn_F, p \right>_F =&0 \quad \mbox{if } F \in \cE_{\partial}\setminus \Gamma_K,\\
	\left< (  \bsigma_h -\nabla u_h) \cdot \bn_F, p \right>_F =&\dfrac{\beta}{h_K}\left<g_h -u_h, p\right>_{\Gamma_K} \quad \mbox{if } F = \Gamma_K.
\end{split}
\end{equation}

Then applying Cauchy-Schwartz inequality and \cref{residual-estimate} gives, for $F\in \cE_I$,
\begin{equation}\label{efficiency-d}
\begin{split}
	&b_F(\theta_h,p) = \sum_{N\in \cN_F}b_F(\theta_{N}, p) 
 \lesssim  h_F^{-1/2} \left(\sum_{N\in \cN_F}\|\{\tilde{\theta}_{N}\}\| \right)  \|p\|_F\\
	&\lesssim  
	\bigg(\underset{F' \in \underset{N\in \cN_F}{\bigcup}\cE_{N} \cap \cE_I}{\sum} \| \jump{ \partial_{\bn_{F'}} u_h}\|_{F'}
	+
	\sum_{K \in\underset{N\in \cN_F}{\bigcup}\cT_{N} } \left( h_K^{1/2}\|f\|_{K\cap\Og_h} 
	+h_K^{-1} \|g_h - u_h\|_{\Gamma_K}  \right)\bigg) \|p\|_F.
\end{split}
\end{equation}
Combining \cref{efficiency-c} and \cref{efficiency-d} gives, for any $F\in \cE$, 
\begin{equation}\label{efficiency-e}
\begin{split}
	&\|(\bsigma_h-\nabla u_h) \cdot \bn_F\|_F \le 
	\sup_{p \in \mathbb P_1(F)} \frac{\left< (  \bsigma_h -\nabla u_h) \cdot \bn_F, p \right>_F }{ \|p\|_F}\\
	\lesssim & 
	\sum_{F' \in \underset{N\in \cN_F}{\bigcup}\cE_{N} \cap \cE_I} \| \jump{ \partial_{\bn_{F'}} u_h}\|_{F'}
	+ \sum_{K \in\underset{N\in \cN_F}{\bigcup}\cT_{N} } \left( h_K^{1/2}\|f\|_{K\cap\Og_h} 
	+h_K^{-1} \|g_h - u_h\|_{\Gamma_K}  \right).
\end{split}
\end{equation}
By the definition of $\bsigma_h$ in \cref{flux-construction} and Cauchy-Schwartz inequality, we also have
\begin{equation}\label{efficiency-f}
\begin{split}
\| \Pi_0(\bsigma_h-\nabla u_h) \|_K 
=& \sup_{p \in \mathbb{P}_1(K)}\frac{(\bsigma_h-\nabla u_h, \nabla p)_K }{ \|\nabla p\|_K}\\
	\lesssim&
	\sum_{F \in \cE_{g} \cap \cE_K } h_F^{1/2} \| \jump{\partial_{\bn_F} u_h} \|_F + h_K^{-1/2} \|g_h - u_h\|_{\Gamma_K}.
	\end{split}
 \end{equation}
 Combining \cref{efficiency-b-a}, \cref{efficiency-e} and  \cref{efficiency-f} , we have
 \begin{equation}\label{efficiency-a}
 \begin{split}
&\| \bsigma_h-\nabla u_h\|_K \\
\lesssim&
\sum_{N \in \cN_K}\left(\sum_{F \in  \cE_N \cap \cE_I} h_F^{1/2}\| \jump{ \partial_{\bn_F} u_h}\|_F+ \sum_{K' \in \cT_N}h_{K'} \|f\|_{K'\cap\Og_h} 
+ h_{K'}^{-1/2} \|g_h - u_h\|_{\Gamma_{K'}} \right).
\end{split}
\end{equation}

For regular facets and elements, there holds the following classical local efficiency results (see \cite{verfurth1994posteriori,cai2017residual}):
	\begin{equation}\label{effi-classical}
	\begin{split}
	h_F^{1/2}\| \jump{ \partial_{\bn_F} u_h}\|_F &\lesssim \| \nabla (u -u_h)\|_{K_F^+ \cup K_F^-} + \|h_K(f - \Pi_1(f))\|_{K_F^+ \cup K_F^-}
	\quad \forall F \in \cE_I \setminus \cE_g,\\
	h_K \|f\|_K &\lesssim \| \nabla (u -u_h)\|_{K} +h_K\|f - \Pi_1(f)\|_{K} \quad \forall K \in \cT_h \setminus \cT_h^b.
	\end{split}
	\end{equation}
\cref{efficiency-b} is then a direct consequence	of \cref{efficiency-a}, \cref{effi-classical} and of the hypothesis $f=\Pi_1(f)$ on any $K \in \cT_h \setminus \cT_h^b$.
This completes the proof of the lemma.
\end{proof}

The following lemma, which follows from theorem 4.5 in \cite{EburmanChe-2020}, gives the efficiency result for the irregular error terms. 
Define
\[
 \mbox{osc}(f) = \left( \sum_{K \in \cT_h^b} h_K^2 \left( \| f - f_K\|_{\omega_K \cap \Omega_h}^2  
 + \|f\|_{ (\Omega_h \setminus \Omega) \cap K}^2\right)\right)^{1/2},
\]
where $\omega_K$ is the union of all elements sharing a common vertex with $K$ and 
$f_K = \underset{c \in R}{\mbox{argmin}} \, h_K\| f -
c\|_{\omega_K \cap \Omega_h}$.
\begin{lemma}\label{lem:best-approximation}
We have the best approximation result for the irregular terms
	\begin{equation}\label{effi-for-irregular-a}
	\begin{split}
	&j_h(u_h,u_h) + \sum_{K \in \cT_h^b} h_K^2 \|f\|_{K\cap \Og_h}^2 +
	\sum_{K \in \cT_h^b} h_{K}^{-1}\| u_h - g_h\|^2_{\Gamma_{K}} \\
	\le&  C_e\,
	 \inf_{v_h \in CG_h} \left( \tn u - v_h \tn^2 + j(v_h,v_h) + \sum_{K \in \cT_h^b} h_{K}^{-1}\| v_h - g_h\|^2_{\Gamma_{K}}+\mbox{osc}(f)^2
	 \right),
	\end{split}
	\end{equation}
where the constant $C_e$ does not depend on the mesh size nor the domain-mesh intersection and $\mbox{osc}(f)$ can be regarded as a higher order oscillation term.
\end{lemma}

%\begin{theorem}
%	
%\end{theorem}

%\textcolor{red}{add the osc(f)}.

%\subsection{Optimal error estimate for the reconstructed flux}

Thanks to the previous result, we can easily deduce an error bound for the flux error $ \|\bsigma-\bsigma_h\|_{\Omega}$, where  $\bsigma=\nabla u $. 

\begin{lemma}\label{thm:ConvHFlux}
Assume $u\in H^{2}(\Omega)$. Then one has:
\begin{equation}\label{ConvHFlux}
\begin{split}
&\|\bsigma-\bsigma_h\|_{\Omega} \lesssim | u- u_h|_{1,\Og} + h|u|_{2,\Og}  \\
&+C_e\,
	 \inf_{v_h \in CG_h} \left( \tn u - v_h \tn^2 + j(v_h,v_h) + \sum_{K \in \cT_h^b} h_{K}^{-1}\| v_h - g_h\|^2_{\Gamma_{K}}+\mbox{osc}(f)^2 \right)
\end{split}
\end{equation}
\end{lemma}

\begin{proof}
We have, using that $\Omega\subset \Omega_h$, that
\[	
\| \bsigma - \bsigma_h\|_\Og \le \|\nabla u- \nabla u_h\|_{\Og} + \| \bsigma_h-\nabla u_h \|_{\Og_h}.
\]
It is therefore sufficient to bound $\| \bsigma_h-\nabla u_h\|_{K} $ for any $K\in\cT_h$. 
For this purpose, we use \cref{efficiency-a}. 
The triangle inequality together with norm equivalence  in a discrete space and standard interpolation results give, for any $F\in \cE_I\setminus \cE_g$, that: 
\begin{equation}\label{optimal-c_2}
\begin{split}
 h_F^{1/2} \| \jump{ \partial_{\bn_{F}} (u-u_h)}\|_{F}  &\lesssim  h_F ^{1/2}\| \jump{ \partial_{\bn_{F}} (u-R_h u)}\|_{F} +  \vert R_h u-u_h\vert_{1,K_F^+\cup K_F^- }\\
 & \lesssim \vert u-u_h\vert_{1,K_F^+\cup K_F^- } +\vert u-R_h u\vert_{1,K_F^+\cup K_F^- } +h \vert u\vert_{2,K_F^+\cup K_F^- } \\
 &\lesssim  \vert u-u_h\vert_{1,K_F^+\cup K_F^- } + h \vert u\vert_{2,K_F^+\cup K_F^- },
 \end{split}
\end{equation}
where $R_h$ is the continuous, piecewise linear  Lagrange interpolation operator.
Together  with \cref{efficiency-a}, the second estimate of \cref{effi-classical} and \cref{lem:best-approximation}, this gives \cref{ConvHFlux} which completes the proof of the lemma. Note that in the proof, we do not need the requirement that $f$ is piecewise linear.
\end{proof}
\begin{remark}

When $\Omega = \Omega_h$, we refer to \cite{BH12} for the a priori error estimate of $\|\nabla( u -u _h)\|_{\Omega} \lesssim h |u|_{2,\Omega}$. In the case when $\Omega \neq \Omega_h$, the same order can be achieved but with some additional inconsistency error of higher order regarding the geometry approximation, which can be  bounded using similar techniques to \cite{EburmanChe-2020}.
\end{remark}

\section{Numerical results}\label{sec:6}
In this section, we present several numerical examples to validate the performance of the a posteriori error estimator in the adaptive mesh refinement procedure.  The adaptive mesh refinement procedure is set as follows:
\[
	\mbox{Solve} \rightarrow \mbox{Estimate} \rightarrow  \mbox{Mark} \rightarrow \mbox{Refine} \rightarrow \mbox{Solve}.
\]
For the penalty parameters in the finite element method, we set $\beta =10$ and $\gamma =0.1$. For the refinement strategy, we use the D\"orfler marking strategy \cite{dorfler1996convergent} and the refinement rate is set to be ten percent.
Regarding the domain approximation,
let $\rho$ be the level set function that satisfies $\rho=0$ on $\partial \Og$ and negative (positive) inside (outside) the domain $\Og$. Let $\rho_h$ be the nodal interpolation of $\rho$ with respect to $\cT_{0,h}$. Then we define 
\begin{equation}\label{domain-approx}
	\partial \Og_h = \{ \bfx: \rho_h(\bfx) =0\}.
\end{equation} 
We can easily check that \cref{eq:geomassum-a} holds.

%\textcolor{red}{Add an example for the convergence of  $\|\nabla u - \bsigma_h\| $ and $\| f - \nabla \cdot \bsigma_h\|$.}

In the adaptive procedure, we compare the error estimators $\eta_1$ and $\eta_2$ defined in \cref{estimator-a} with the residual based error estimator (see \cite{EburmanChe-2020}) defined as follows,
\begin{equation} \label{indicator}
	\eta_{K,res} =   \sqrt{h_K^2\| f \|_{K \cap \Og_h}^2 +
	h_K^{-1} \beta^2 \|g_h-u_h\|_{\Gamma_K}^2 
	%+  \| \nabla \tilde e\|_{K \cap \Og}^2 
	+
	 \sum_{F \in \cE_K \cap \cE_I}\dfrac{h_F}{2} \|\jump{\partial_{\bn_F} u_h}\|_{F}^2}.
\end{equation}
The global residual based error estimator is then defined by
\begin{equation} \label{estimator}
\eta_{res} = \left( \sum_{K \in \cT_h} \eta_{K,res}^2   \right)^{1/2}.
\end{equation}

\begin{example}\label{ex2a}
In this example, we test a problem with a strong interior peak. The exact solution has the following representation:
\[
u(x,y) = \exp( -100 ((x - 0.5)^2 + (y- 0.5)^2)).
\]
This function has a strong peak at the point $(0.5, 0.5)$.
\end{example}
%This function has a strong peak at $(1/2, 1/2)$. %(see  \cref{Fig:Ex1-Franke}).
Note that the boundary of the domain is regular, thus we have that $\eta_{K,1} = \eta_{K,2}$. Moreover, the function value is very smooth and almost vanishes on the boundary. The purpose of this example is to test the efficacy of our adaptive algorithm for Nitsche's method on a regular domain.

In the numerical scheme, $g$ and $f$ are  approximated by their interpolations into the continuous piecewise linear space. 
We firstly test the convergence of the method on uniform meshes. The results are plotted in \cref{Ex2a}(e) which show optimal convergence rates (order $1$) for both the true error $\|\nabla (u-u_h)\|$ and the flux error $\| \nabla u - \bsigma_h\|$.

In the adaptive mesh refinement (AMR) procedure, we start with a $5 \times 5$ initial mesh. The marking percent is set to be $25\%$, i.e., the ordered elements (from the one with largest error indicator) that accounts for the first $25\%$ of the total error estimator get refined. With the stopping criteria that the total number of degree of freedoms (DOFs) be not greater than $5000$, the final meshes generated by $\eta_{1,K}$ and $\eta_{res,K}$ are provided in \cref{Ex2a}(a) and \cref{Ex2a}(b).
\begin{figure}[ht]
\centering
\begin{tabular}{ccc}
{\includegraphics[width=0.30\textwidth]{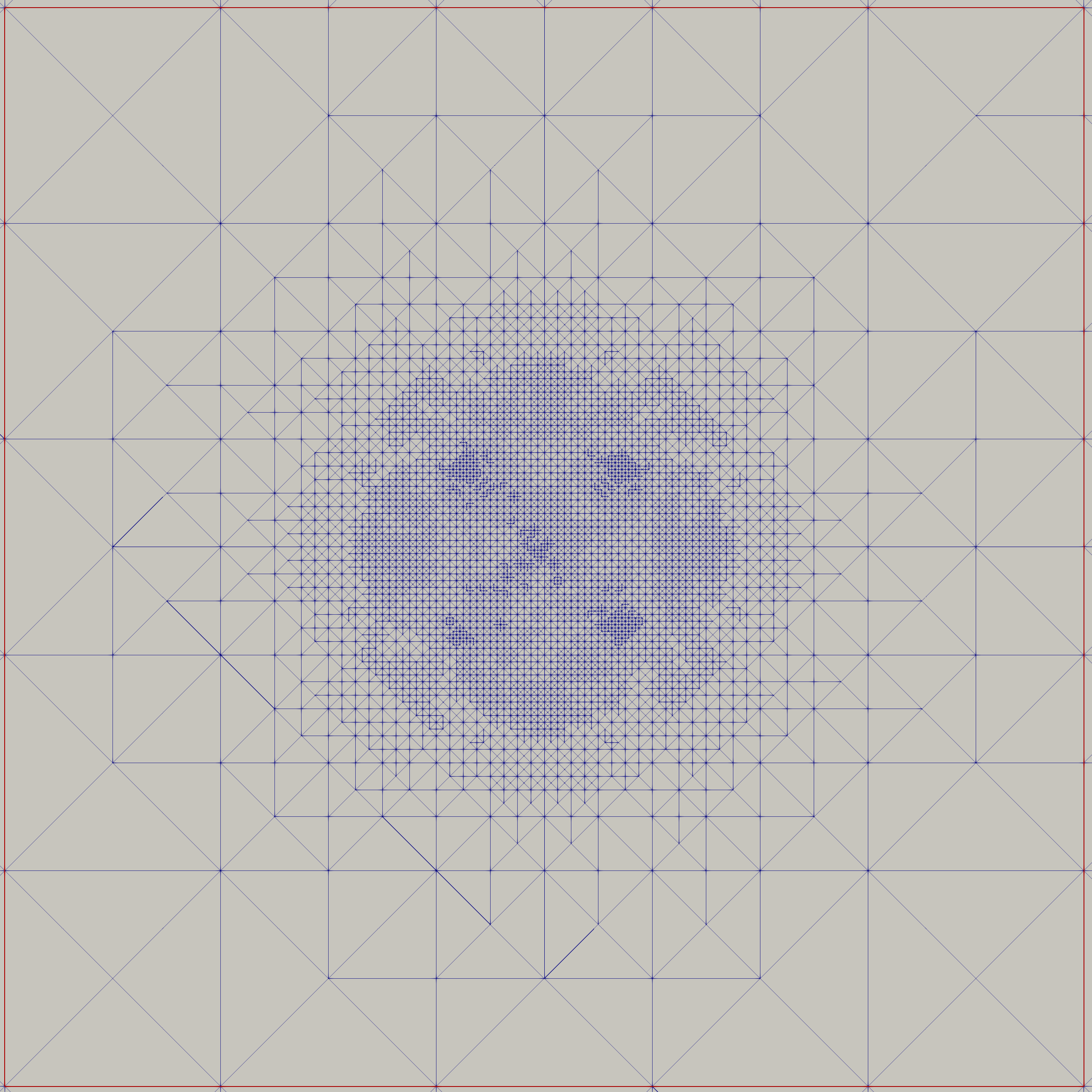} }&
{\includegraphics[width=0.30\textwidth]{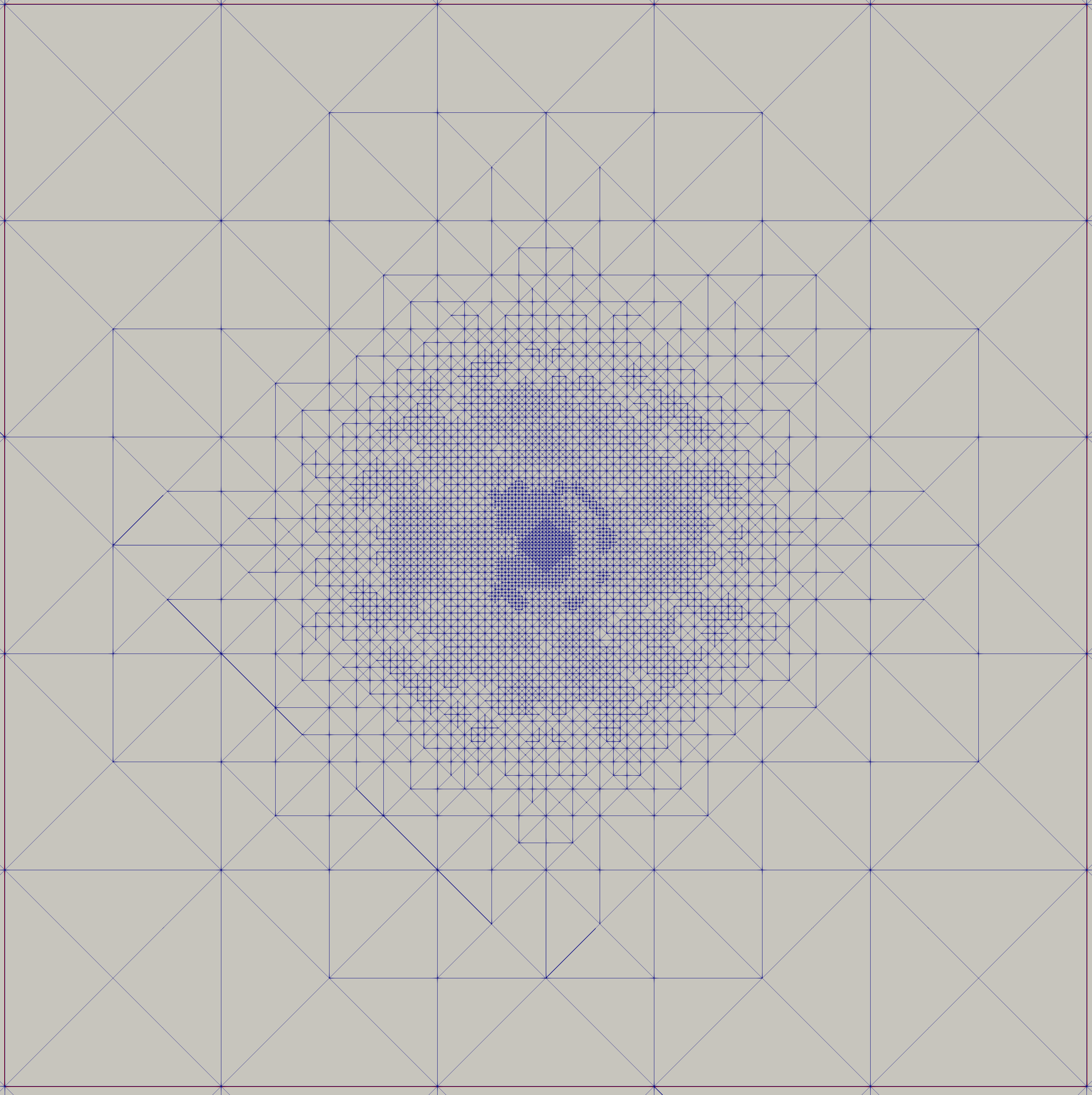}}&\\
(a)mesh by $\eta_{1,K}$ &(b)mesh by $\eta_{res,K}$&\\
{\includegraphics[width=0.30\textwidth]{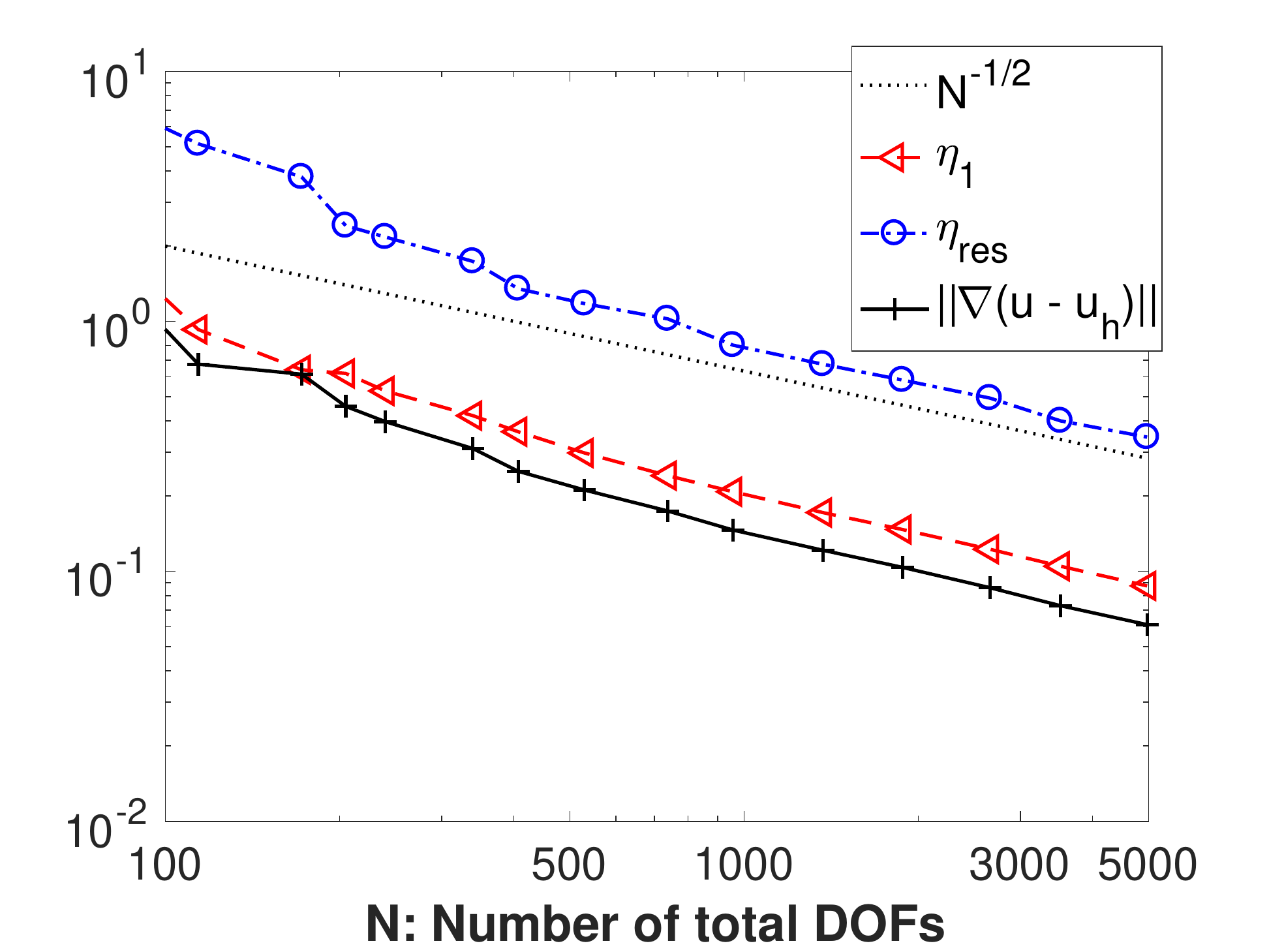} }&
{\includegraphics[width=0.30\textwidth]{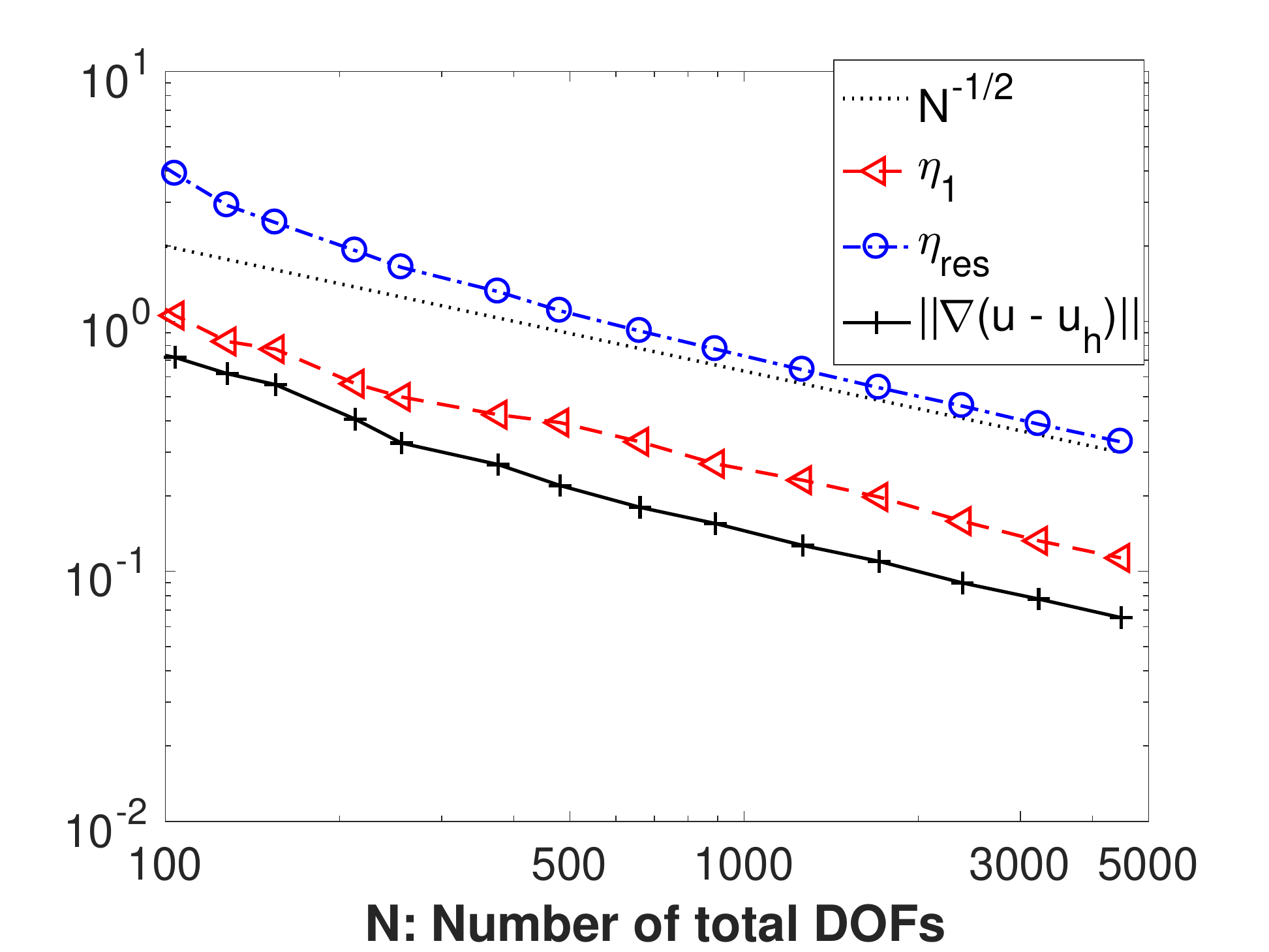}}&
{\includegraphics[width=0.30\textwidth]{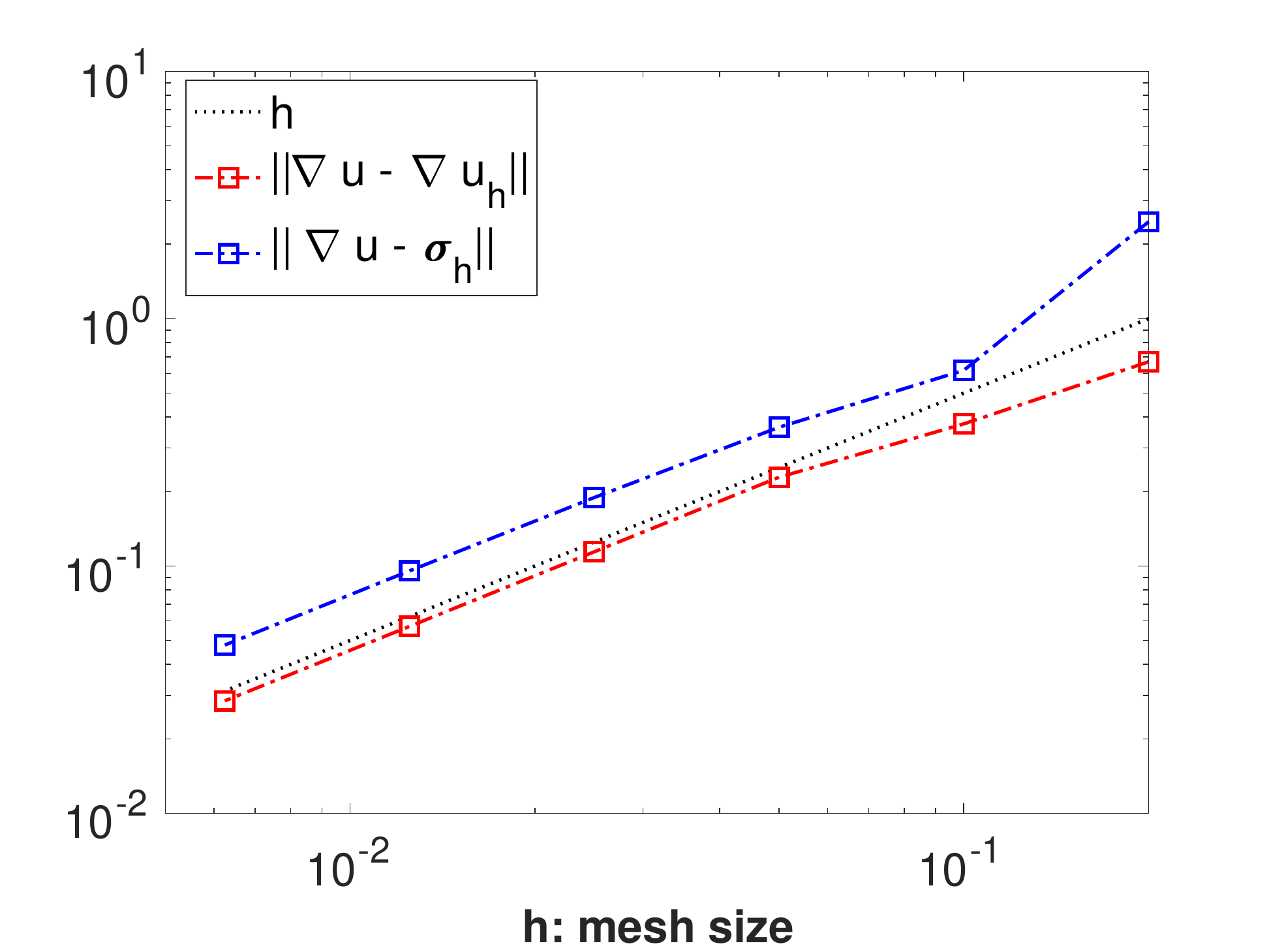}}\\
(c)Errors by $\eta_{1,K}$ &(d)Errors by $\eta_{res}$& (e)On uniform meshes\\
\end{tabular}
\caption{\cref{ex2a}. Final meshes and convergence of error estimators} %using $\eta_{classical}$ and $\eta$
 \label{Ex2a}
\end{figure}
From \cref{Ex2a}(c) and \cref{Ex2a}(d), we observe optimal convergence rates for both estimators. However, the efficiency index, which is defined by $\dfrac{\eta}{\| \nabla (u -u_h)\|}$, of $\eta_1$ is more accurate with mean values $1.42$ and $1.68$ for  \cref{Ex2a}(c)  and \cref{Ex2a}(d), respectively,  comparing to that of $\eta_{res}$ with mean values $5.75$ and $5.10$ for  \cref{Ex2a}(c)  and  \cref{Ex2a}(d), respectively.

\begin{example}\label{ex2}
In this example, we test the Franke function \cite{franke1979critical} on the unit square domain,
\begin{equation*}
\begin{split}
	u(x,y)  =&\dfrac{3}{4} \exp{\left(-(9x-2)^2/4 - ( 9y-2)^2/4\right)} 
	+ \dfrac{3}{4} \exp{(-(9x+1)^2/49 - (9y+1)/10)}\\
	&+\dfrac{1}{2}\exp{(-(9x-7)^2/4 - (9y-3)^2/4)}
	-\dfrac{1}{5} \exp{(-(9x-4)^2 - (9y-7)^2)}.
\end{split}
\end{equation*}
This function has two peaks at $(2/9, 2/9)$ and $(7/9,1/3)$ and one sink at $(4/9, 7/9)$. %(see  \cref{Fig:Ex1-Franke}).
Since the boundary of the domain is regular, the purpose of this example is again to test the efficacy of our algorithm for Nitsche's method on regular domain. However, the solution on the boundary is more volatile than in \cref{ex2a} and our numerical results show that this boundary volatility potentially causes extra challenges for the efficiency of Nitsche's method that imposes the Dirichlet boundary condition weakly.
\end{example}
The optimal convergence results on uniform meshes for the errors $\| \nabla (u - u_h)\|$ and $\| \bsigma - \bsigma_h\|$ are verified in \cref{Ex2}(e). 
With the same initial mesh and marking strategy as in \cref{ex2a}, and with the stopping criteria that the total number of DOFs be not greater than $7500$, the final meshes generated using $\eta_{1,K}$ and $\eta_{res,K}$ are provided respectively in \cref{Ex2}(a) and (b). Both meshes are similar with DOFs centered around the peaks and sinks. Since the solution is more volatile on some parts of the boundary, we also observe dense refinements on some right and upper parts of the boundary. However, the mesh in \cref{Ex2}(a) puts relatively more DOFs on the boundary comparing to  \cref{Ex2}(b) on the boundary.

From \cref{Ex2}(c) and \cref{Ex2}(d), we observe optimal convergence for both the true error and the estimators in the overall pattern, however, with occasional oscillations, for both cases. Such oscillation is uniquely caused by the Nitsche method since it imposes the Dirichlet boundary condition weakly. Again, we observe that $\eta_{1}$ is more accurate than $\eta_{res}$ for most regular (non-oscillating) iterations. Nevertheless, it seems that $\eta_1$ has a stronger magnifying effect for the oscillation.   

\begin{figure}[ht]
\centering
\begin{tabular}{ccc}
{\includegraphics[width=0.30\textwidth]{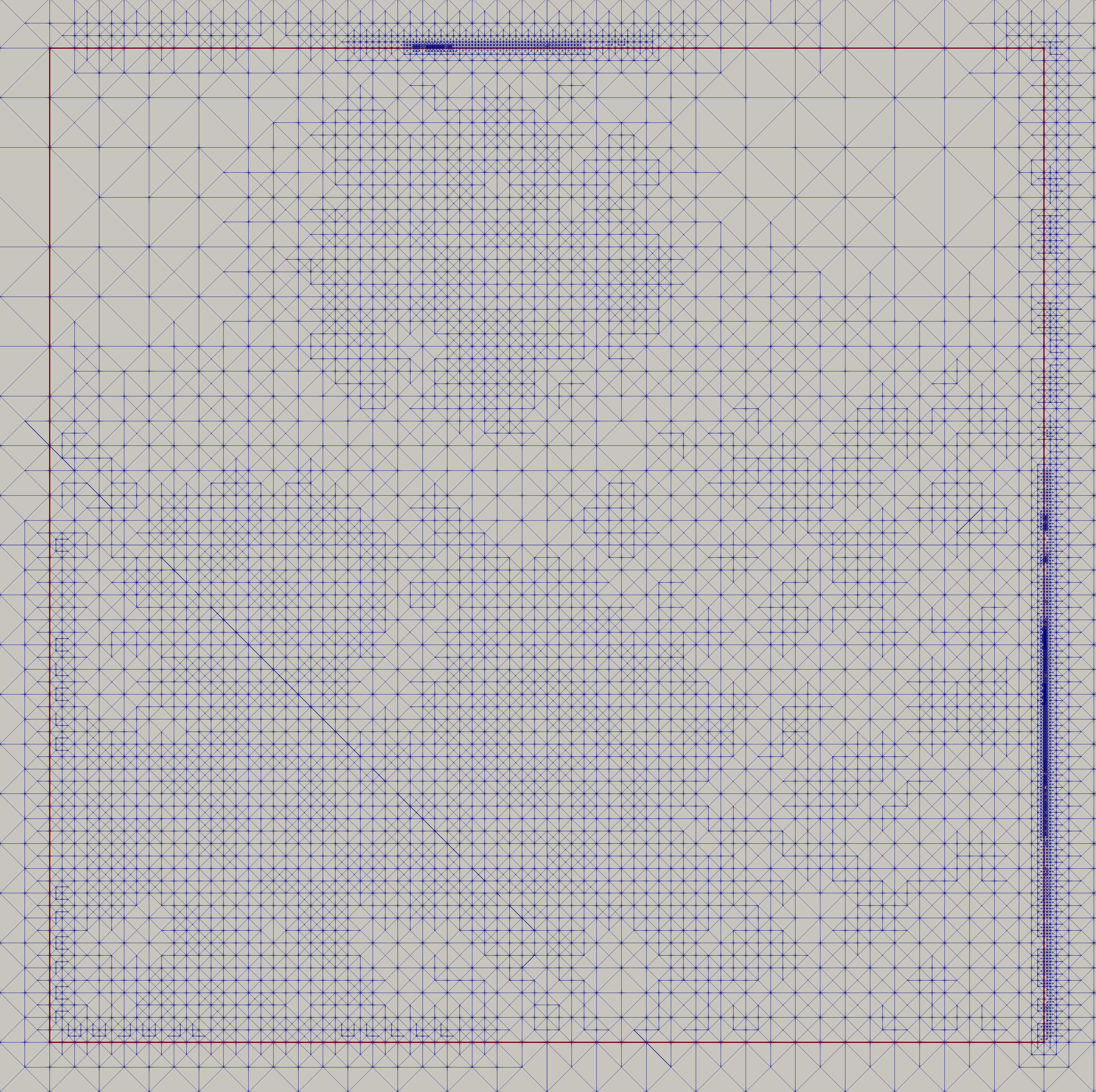} }&
{\includegraphics[width=0.30\textwidth]{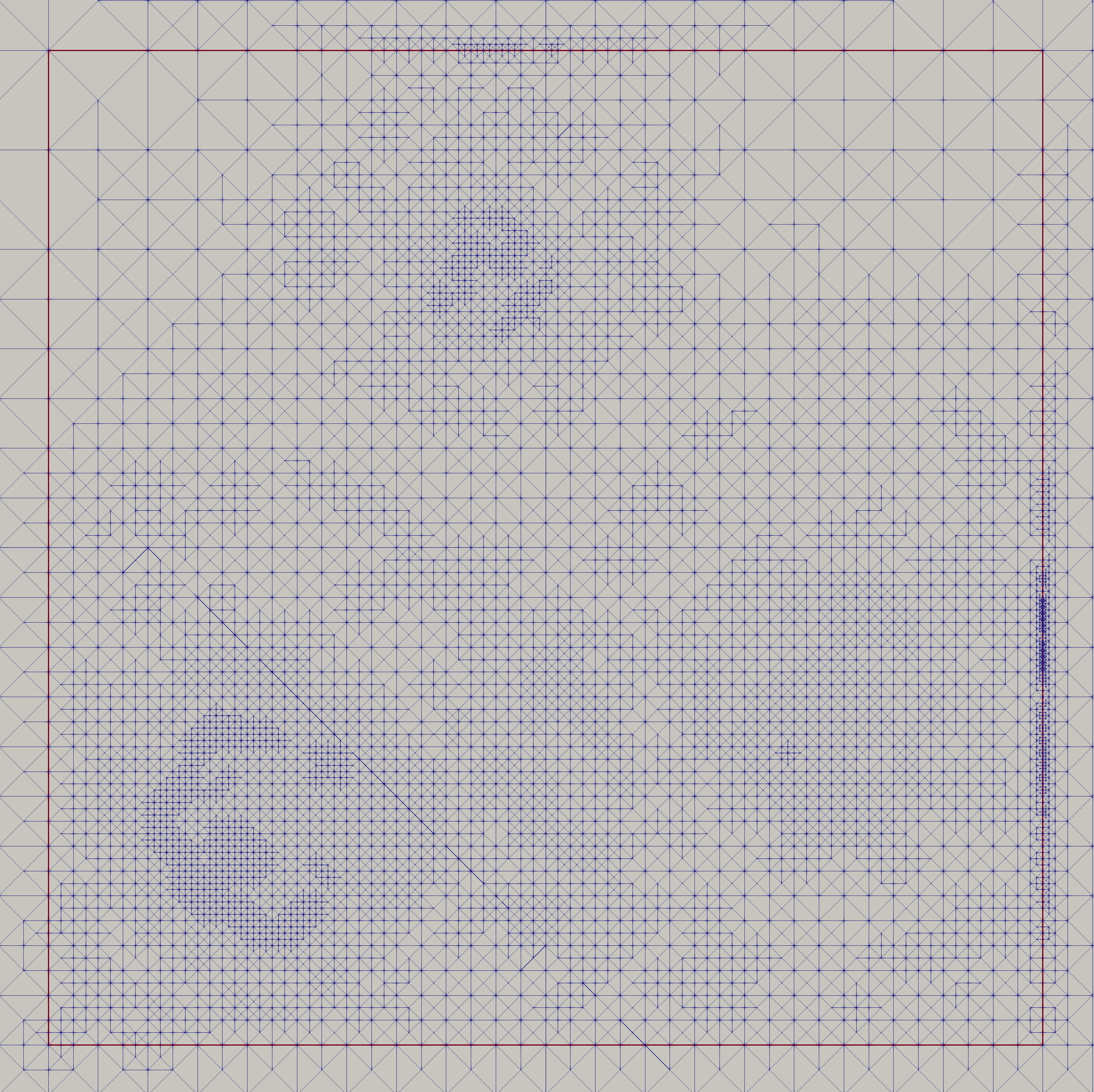}}&\\
(a)Mesh by $\eta_{1,K}$ &(b)Mesh by $\eta_{res,K}$&\\
{\includegraphics[width=0.30\textwidth]{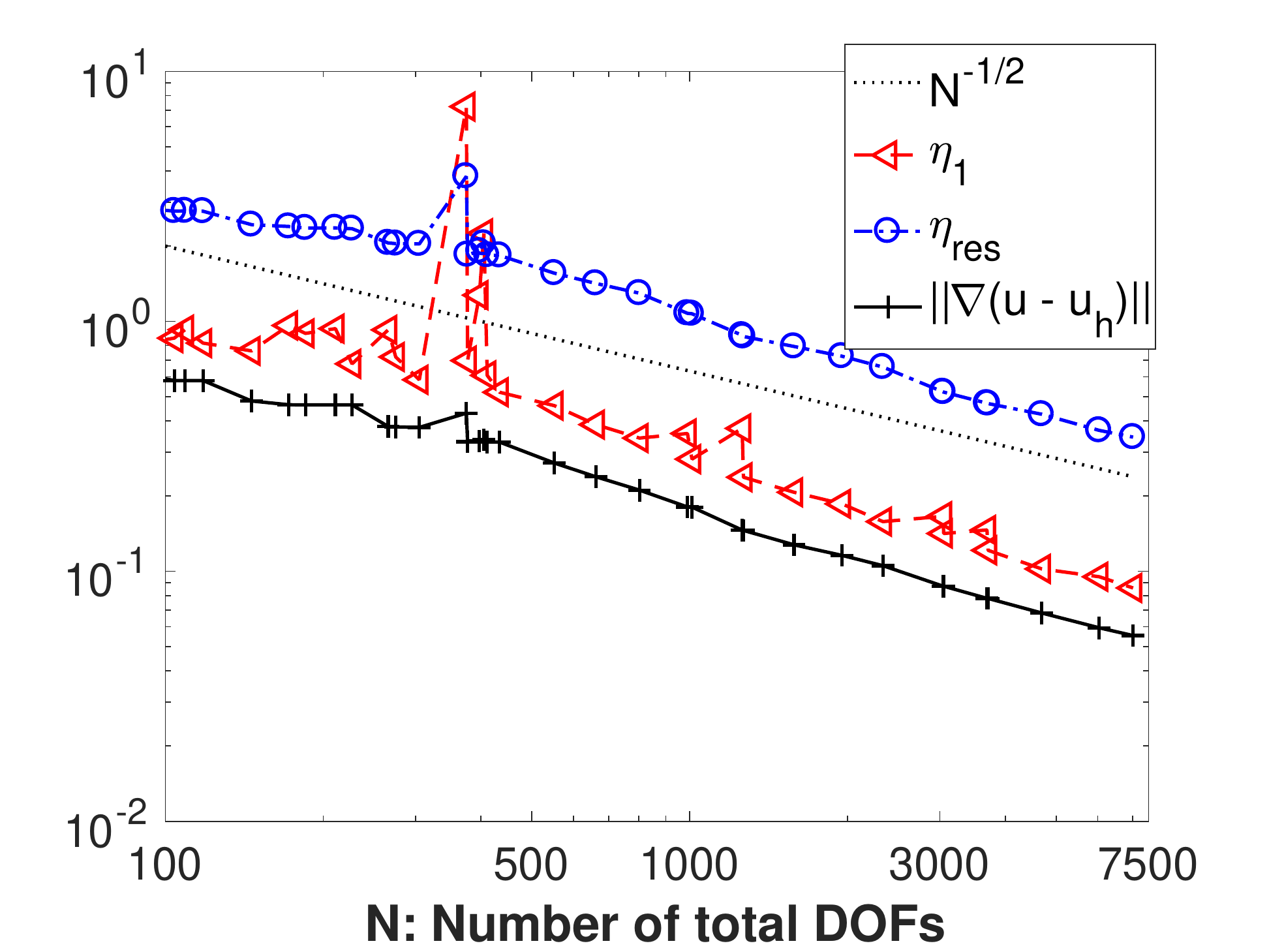} }&
{\includegraphics[width=0.30\textwidth]{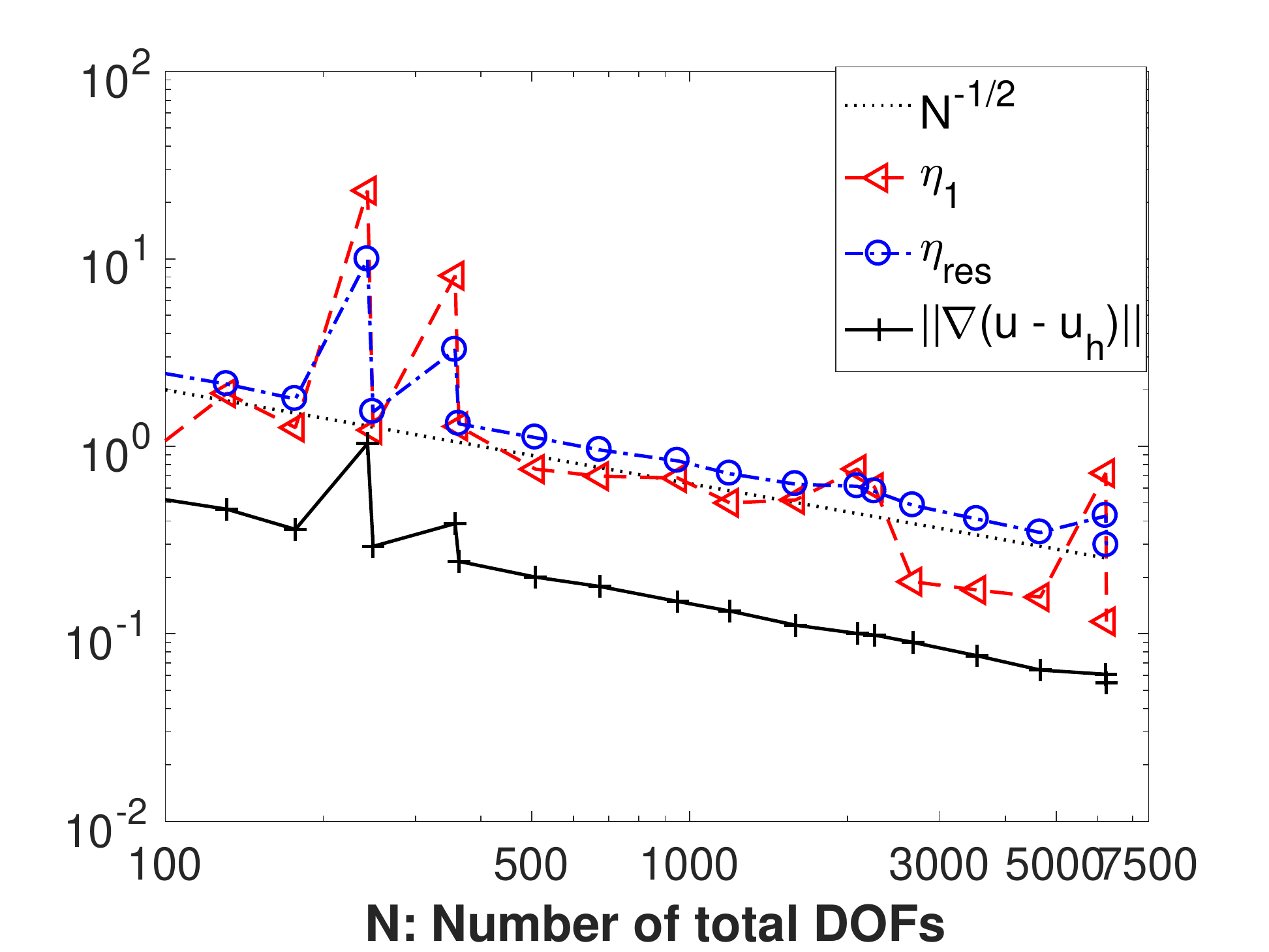}}&
{\includegraphics[width=0.30\textwidth]{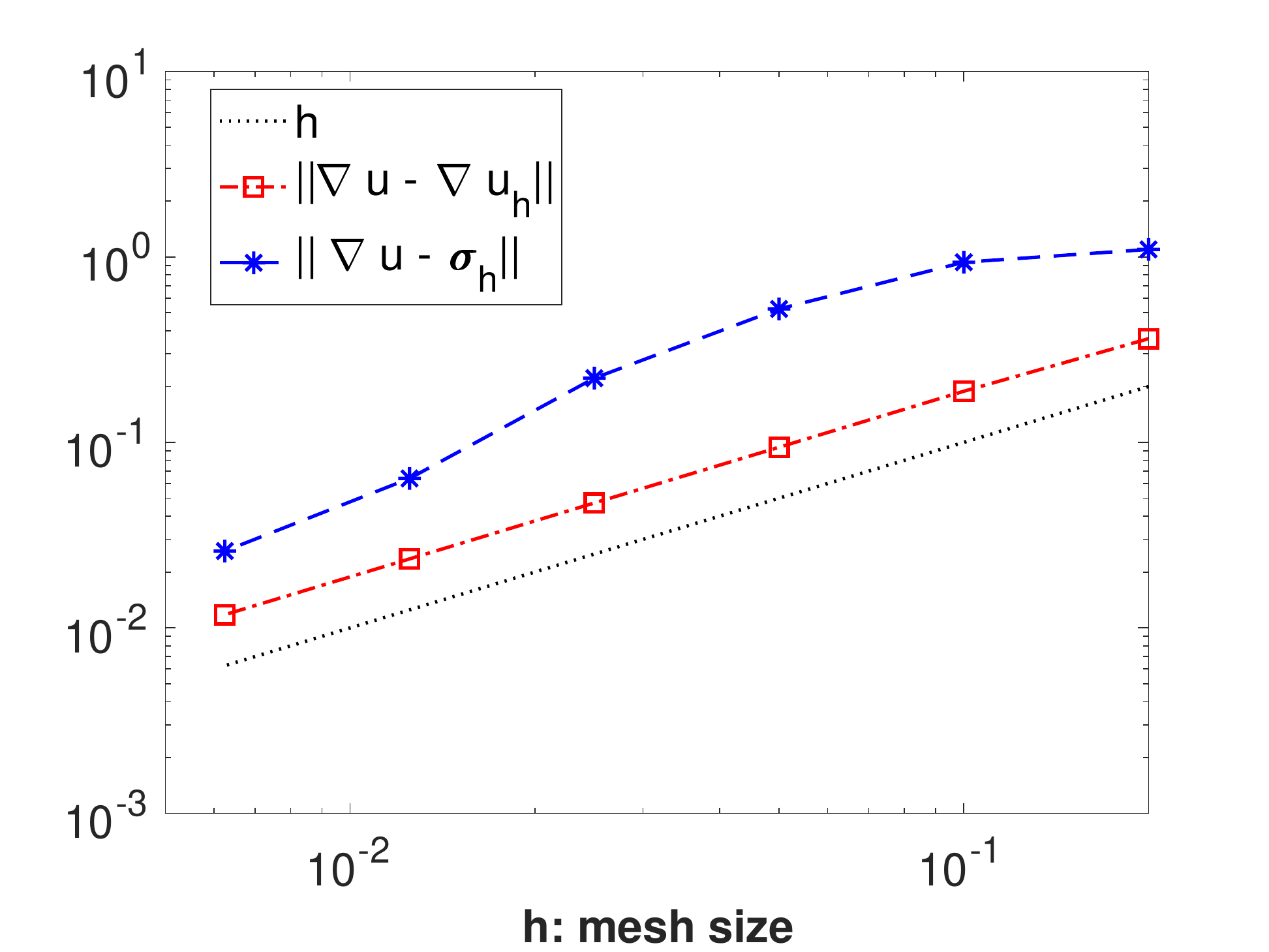}}\\
(c)Errors by $\eta_1$ &(d)Errors by $\eta_{res}$& (e)On uniform meshes\\
\end{tabular}
\caption{\cref{ex2}. Final meshes and convergence of error estimators} %using $\eta_{classical}$ and $\eta$
 \label{Ex2}
\end{figure}

\begin{example} \label{ex1}
In this example, we test our algorithm on an irregular domain.
The level set of the problem has a flower shape (see e.g., \cref{Ex1-no-correction}(a))  that has the 
following representation:
\[
	\rho = \min(\rho_0, \rho_1, \cdots, \rho_8)
\]
with
\[
	\begin{cases}
	\rho_0(x,y)= x^2 + y^2 - r^2, & r=2\\
	\rho_i(x,y) = (x - x_i)^2 + (y - y_i)^2 - r_i^2, &r_i = \sqrt{2} r(\sin( \pi/8) + \cos(\pi/8) ) \sin(\pi/8)
	\end{cases}
\]
for $i = 1, \cdots, 8$, and 
\[x_i = r(\cos(\pi/8)+\sin(\pi/8))\cos(i\pi/4), \quad
y_i = r(\cos(\pi/8)+\sin(\pi/8))\cos(i\pi/4).\]
The domain boundary is defined to be the zero level set, i.e., $\Og = \{ \bfx \in \mathbb{R}^2: \rho(\bfx) \le 0 \}$. 
The data are given such that $g = 0$ on $\partial \Og$ and 

\[   
	f(x,y) =\left\{
	\begin{array}{lll}
      	10 & \mbox{if }  (x - x_1)^2 + (y - y_1)^2 \le r_1^2/2,\\
	0 & \mbox{otherwise}.
\end{array} 
\right. \]
\end{example}

In the numerical scheme, we take $g_h \equiv 0$ and $f$ is  approximated by its $L^2$ projection into the discontinuous piecewise constant space. We start with a $8 \times 8$ crossed mesh on the rectangular domain $(-4,4) \times (-4,4)$.
With the stopping criteria that the total number of DOFs be not greater than $7000$ and marking percent set to be $15\%$, the final meshes obtained by $\eta_1$, $\eta_2$ and $\eta_{res}$ are given in
\cref{Ex1-no-correction}(a), (b) and (c), respectively. We observe similar meshes for the three cases and DOFs are centered around the heat source. From \cref{Ex1-no-correction}(c), (d) and (e), we observe optimal convergence rate for all error estimators. In this example, $\eta_1$ is very close to $\eta_2$ since there are no dense refinement on the boundary, and $\eta_{res}$ is relatively bigger. 
\begin{figure}[ht]
\centering
\begin{tabular}{ccc}
{\includegraphics[width=0.30\textwidth]{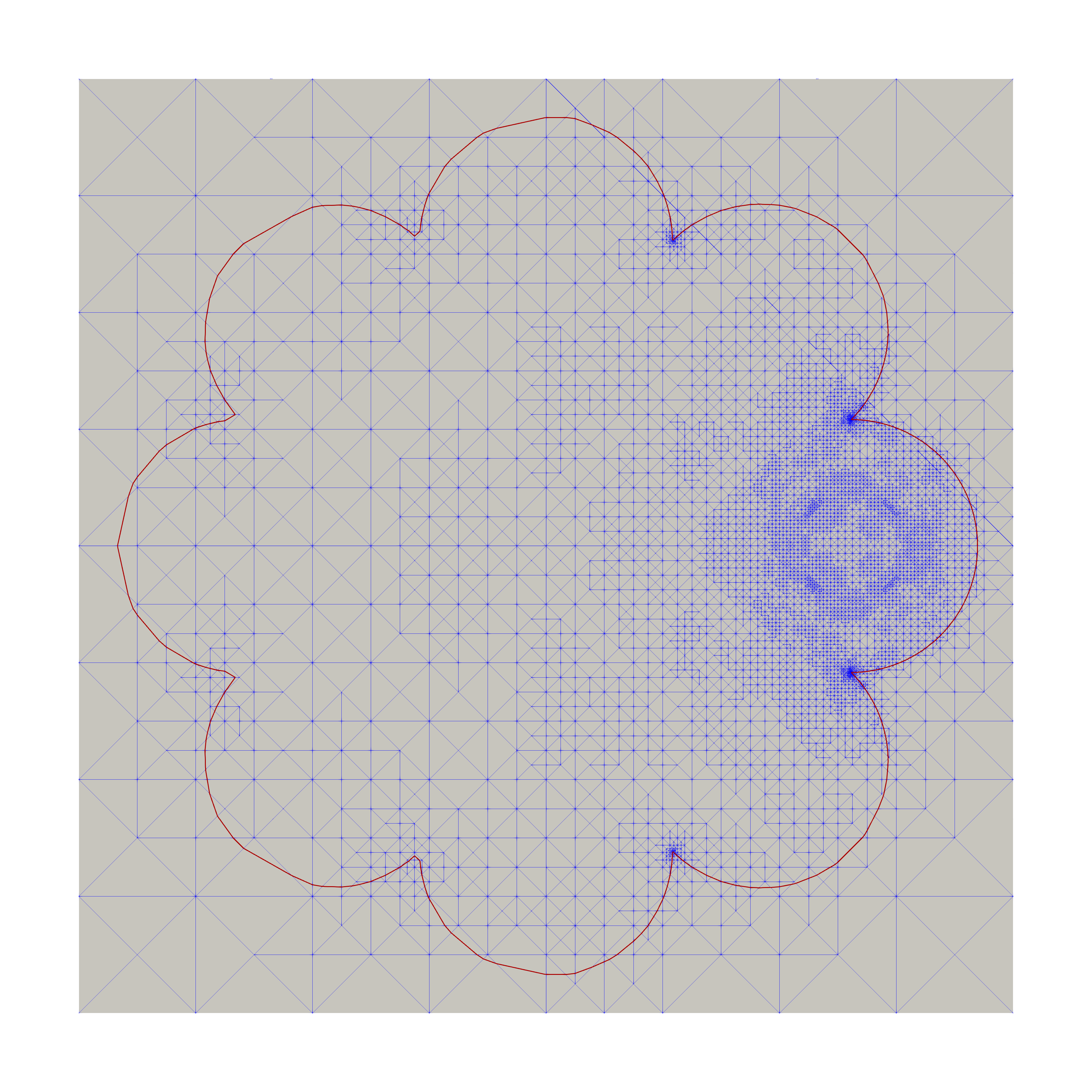} }&
{\includegraphics[width=0.30\textwidth]{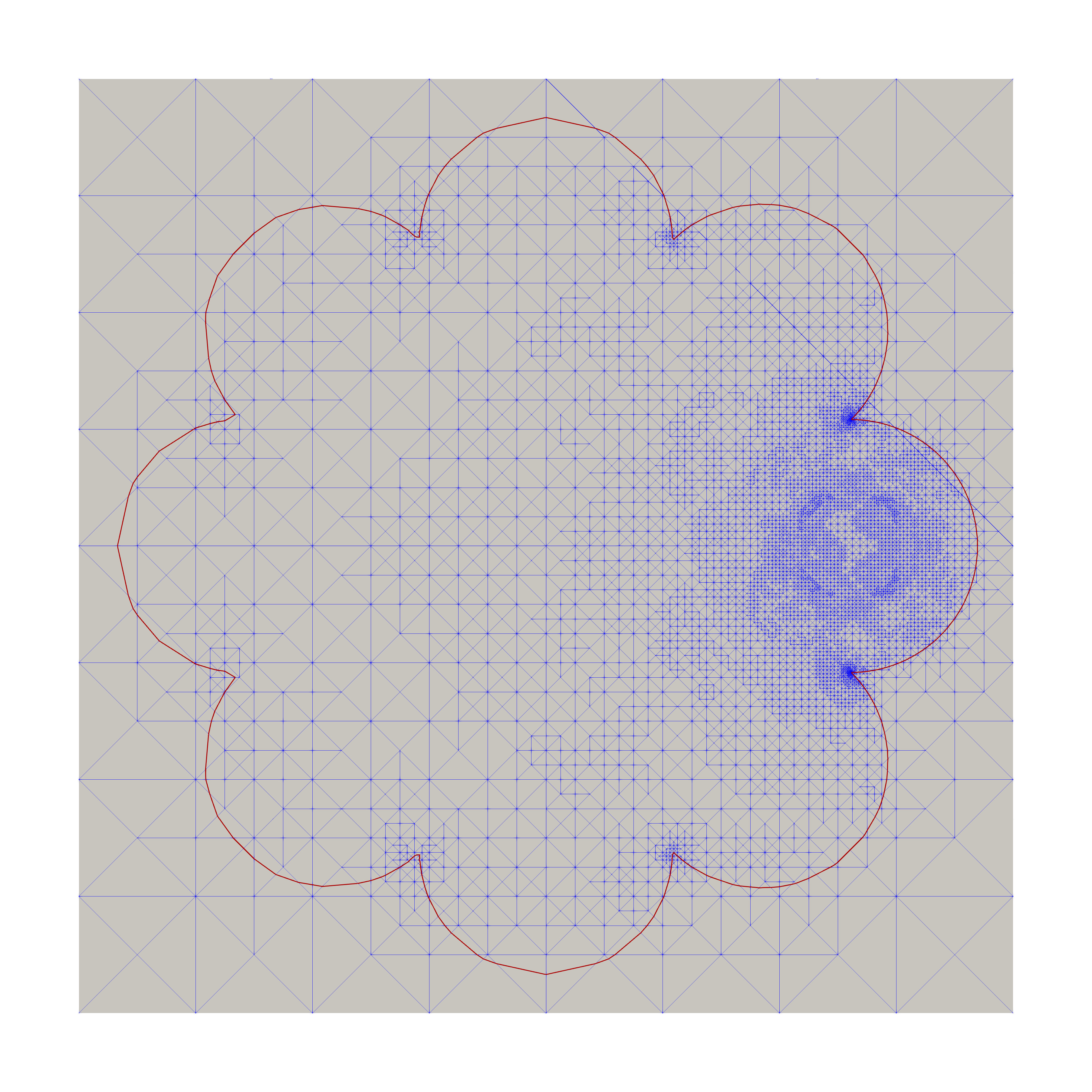}}&
{\includegraphics[width=0.30\textwidth]{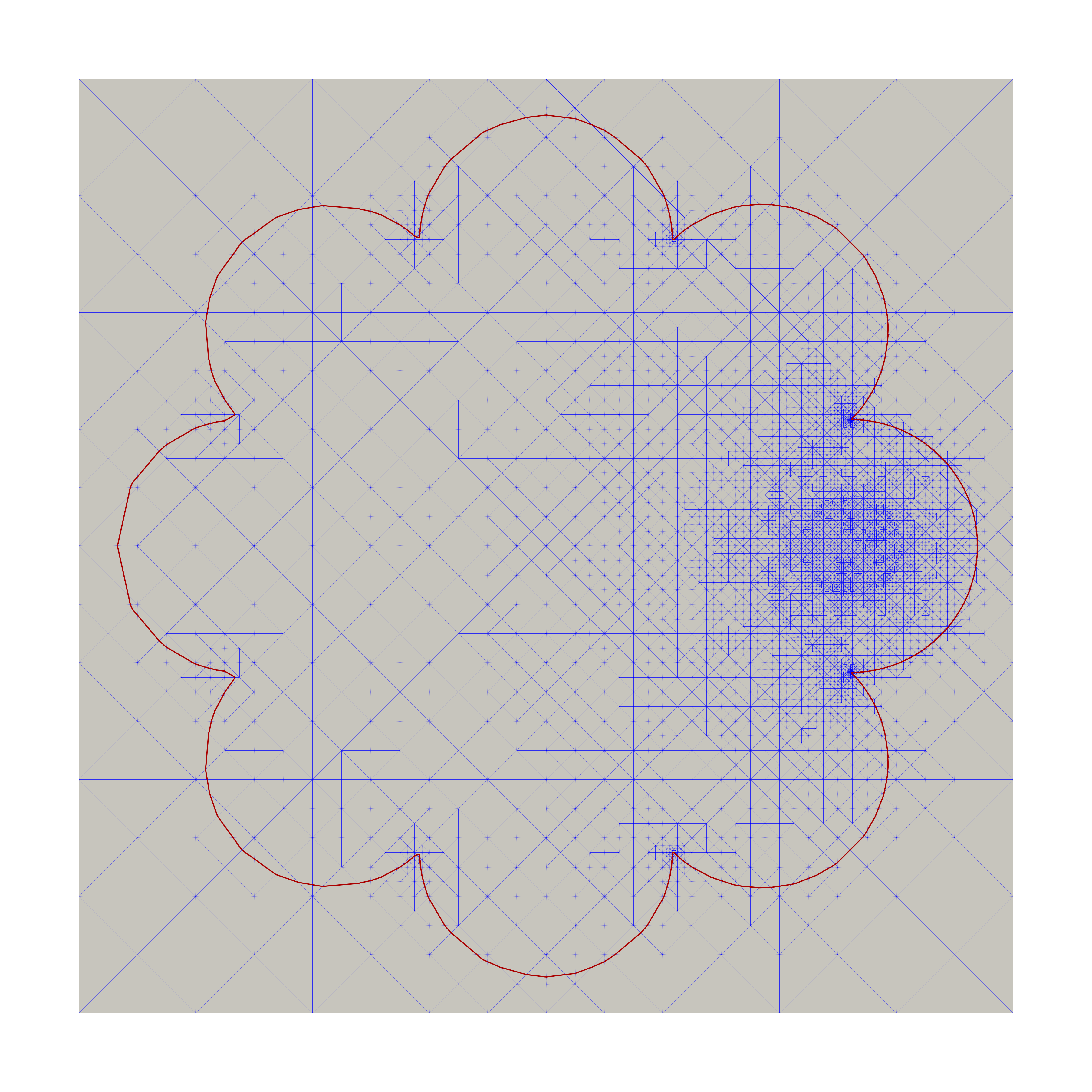}}\\
\\
(a)Meshes by $\eta_{1,K}$ &(b)Meshes by $\eta_{2,K}$& (c)Meshes by $\eta_{res, K}$\\
{\includegraphics[width=0.30\textwidth]{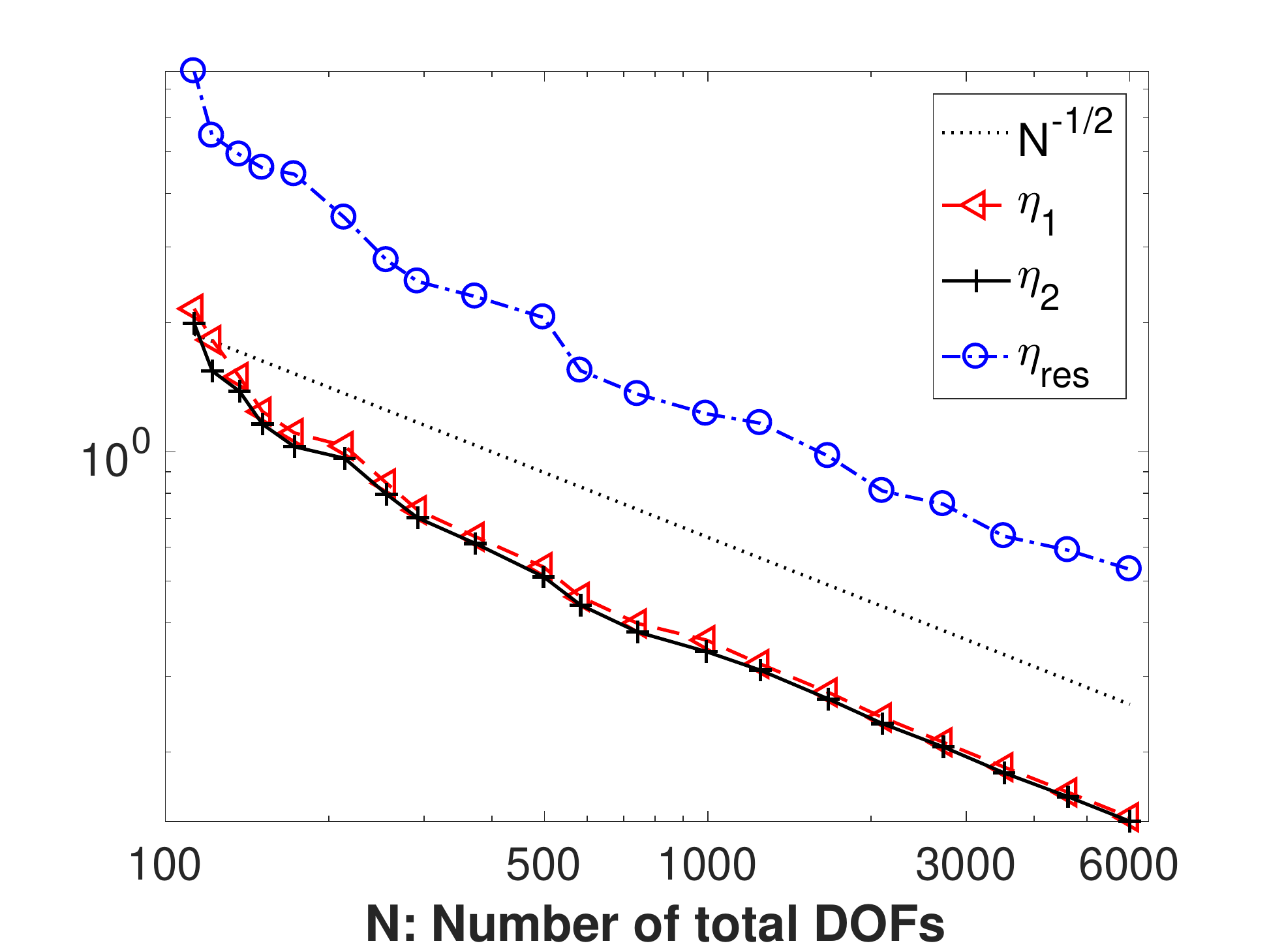} }&
{\includegraphics[width=0.30\textwidth]{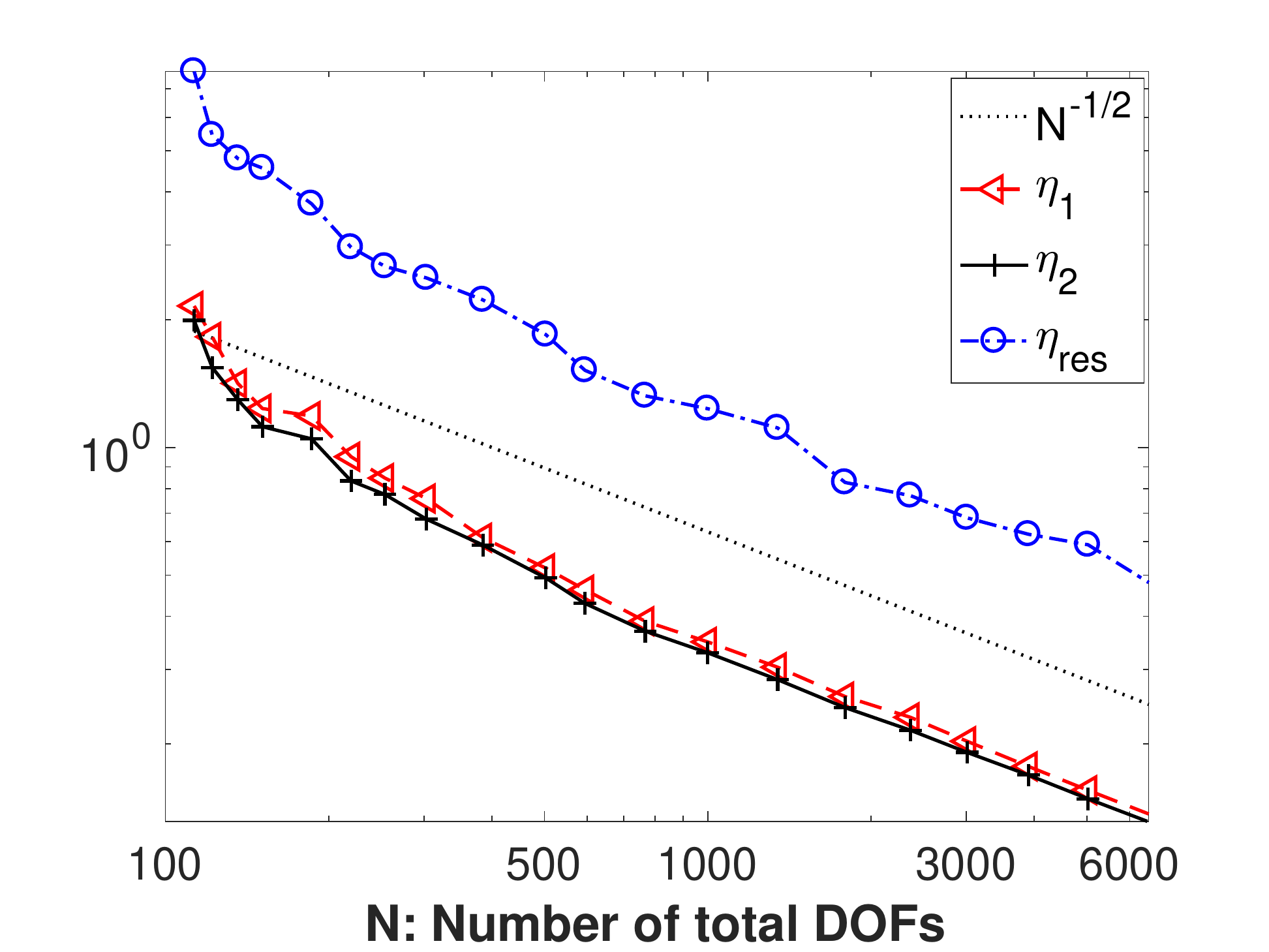}}&
{\includegraphics[width=0.30\textwidth]{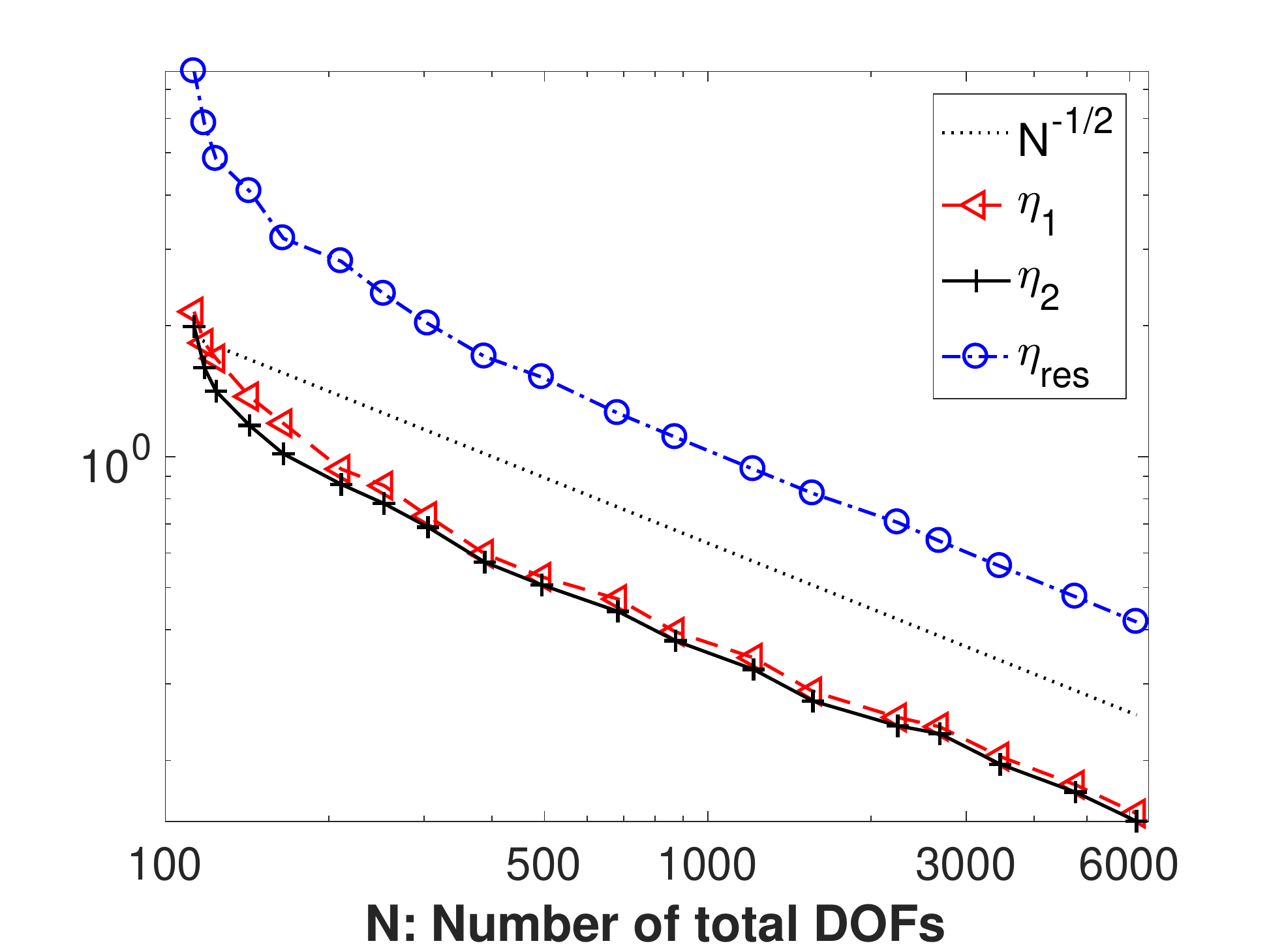}}\\\\
(d)Errors by $\eta_{1,K}$ &(e)Errors by $\eta_{2,K}$& (f)Errors by $\eta_{res, K}$\\
\end{tabular}
\caption{\cref{ex1}. Final meshes and convergence of error estimators} %using $\eta_{classical}$ and $\eta$
 \label{Ex1-no-correction}
\end{figure}
%We also test the error estimators $\eta_1$ and $\eta_2$ without adding the boundary correction term. The meshes and error convergence results are presented in \cref{Ex1-no-correction}. 
%We observe that the boundary DOFs around the heat pedal disappeared and the generated meshes  are very similar to the one generated by the $\eta_{res}$ as in \cref{Ex1}(c). 
%Also, in both \cref{Ex1-no-correction}(c) and  \cref{Ex1-no-correction}(d), we observe optimal convergence for all estimators and that $\eta_1$ and $\eta_2$ are almost equal for all iterations in both AMR procedures. Based on these results, removing the boundary correction term when the mesh is not regular, i.e., $\Gamma_K\not \subset \cE_K$, does not affect, if not improve, the convergence performance.

\begin{example}\label{ex3}
In this example, we consider the reentrant problem whose solution has the following polar representation:
\[
	u(r, \theta) = r^{\alpha} \sin(\alpha \theta),
\]
with $\alpha = \pi/ \omega$ and $\omega$ being the angle of the reentrant corner. In
this example, we take $\omega = 3/2 \pi$. The domain is set to be $\Omega = ([-1,1]^2 \setminus [0,1]\times[-1,0] )\cap B(0.95)$, where $B(0.95)$ is the ball with center $(0,0)$ and radius $0.95$. It is easy to check that $f =0$ in $\Omega$.
\end{example}
In the numerical scheme, we extend $f$ outside of $\Omega$ by $0$ and take $g_h$ to be the conforming linear interpolation of $u$ with respect to $\cT_{h}$.
The optimal convergence results on uniform meshes for the errors $\| \nabla (u - u_h)\|$ and $\| \bsigma - \bsigma_h\|$ are verified in \cref{Ex3}(e). 

In the AMR procedure, we choose to use the initial mesh $10 \times 10$ on the regular domain $(-1,1) \times (-1,1)$.
 With the marking percent set to be $10\%$ and the stopping criteria set such that the maximal number of degrees of freedom does not exceed $5000$, the final meshes generated by $\eta_{2,K}$ and $\eta_{res}$ are given in 
\cref{Ex3}(a) and \cref{Ex3}(b).  
The corresponding convergence rate of the estimators are presented in 
\cref{Ex3}(c) and \cref{Ex3}(d).
We again observe optimal convergence for both AMR procedures. This again indicates that the estimators work equivalently effective for problems with reentrant singularity on the boundary. However, $\eta_{2}$ is more accurate than $\eta_{res}$. 
For \cref{Ex3}(c), the mean ratio for  $\eta_{res}/\|\nabla (u - u_h)\|_{\Og}$ is $ 4.1$,
whereas the mean ratio for $\eta_1/\|\nabla (u - u_h)\|_{\Og}$ and $\eta_2/\|\nabla (u - u_h)\|_{\Og}$  is  $2.4$ and $1.5$, respectively. The corresponding ratios for \cref{Ex3}(d) are similar.
We note that for this example, using $\eta_{1,K}$ generates almost the same mesh as $\eta_{2,K}$. However, $\eta_2$ is more accurate than $\eta_1$ in both cases.

\begin{figure}[ht]
\centering
\begin{tabular}{ccc}
{\includegraphics[width=0.30\textwidth]{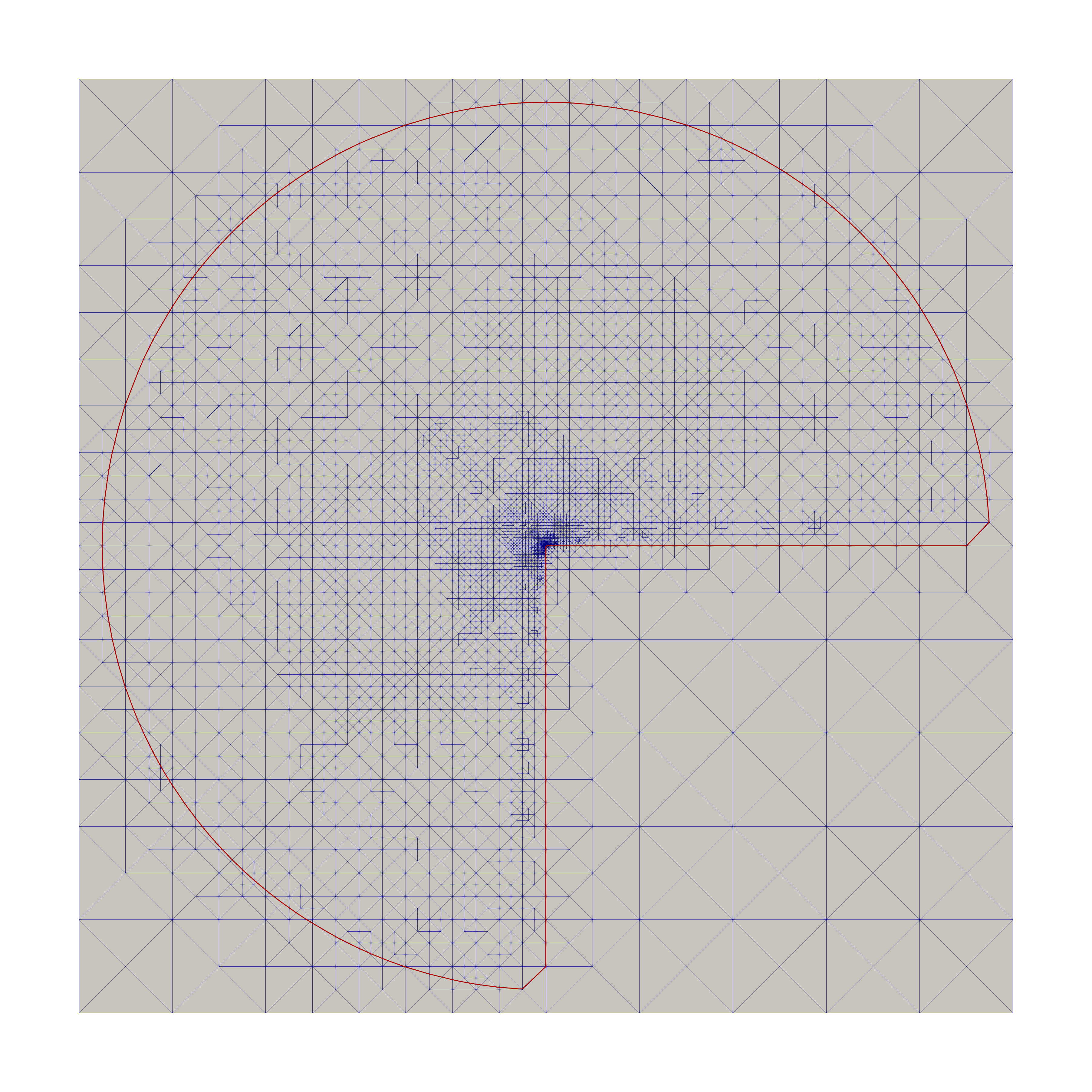} }&
{\includegraphics[width=0.30\textwidth]{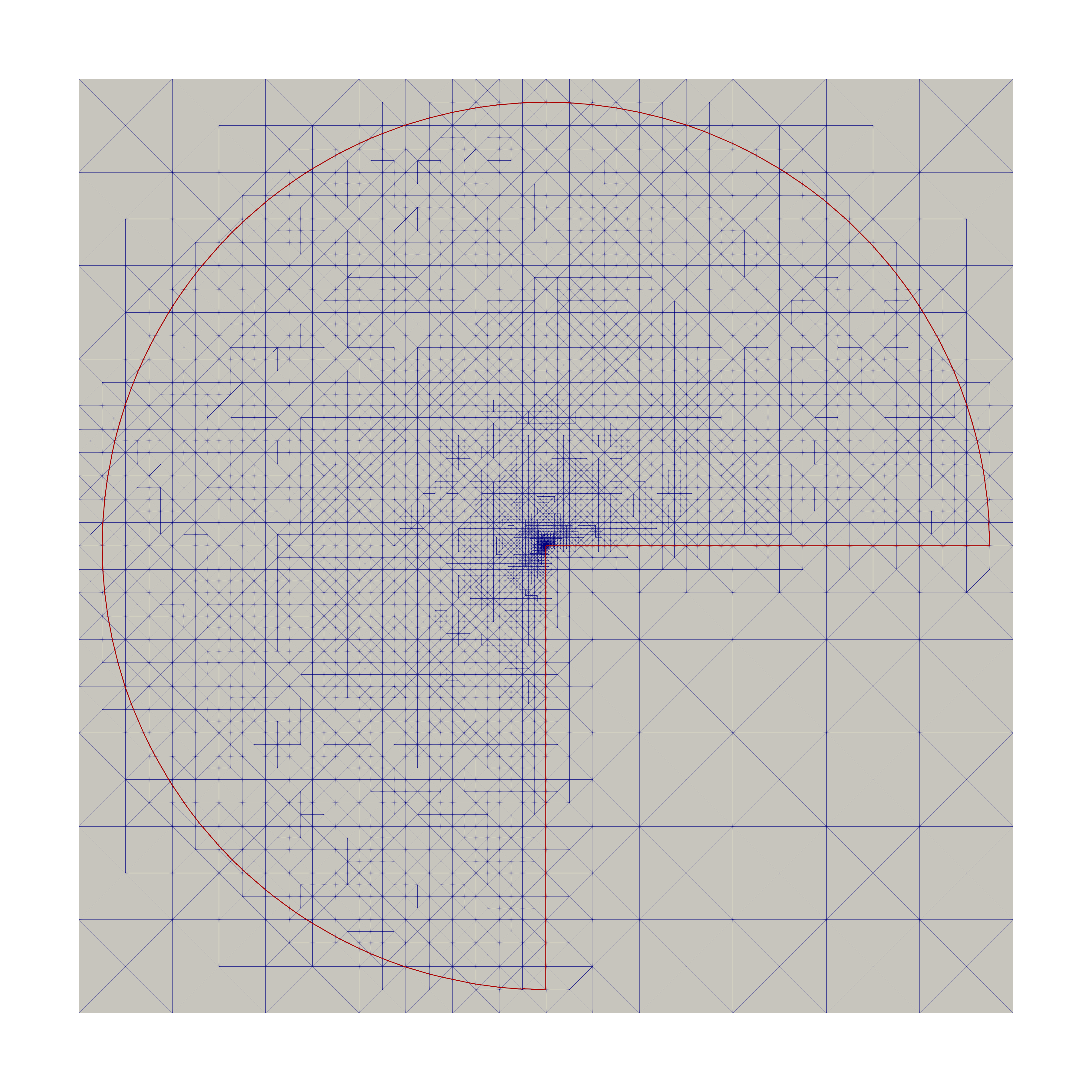}}&\\
(a)Mesh by $\eta_{2,K}$ &(b)Mesh by $\eta_{res,K}$&\\
{\includegraphics[width=0.30\textwidth]{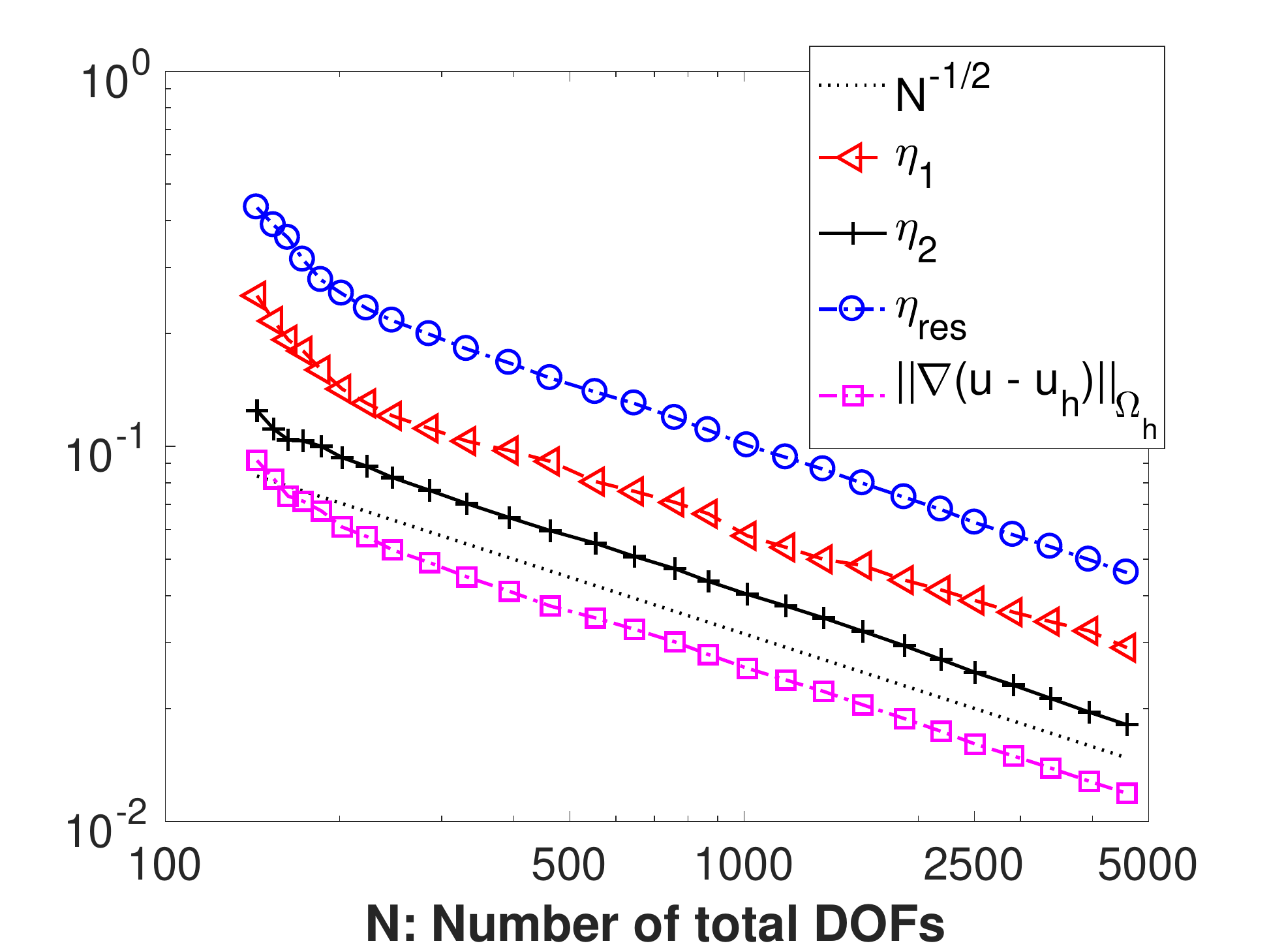} }&
{\includegraphics[width=0.30\textwidth]{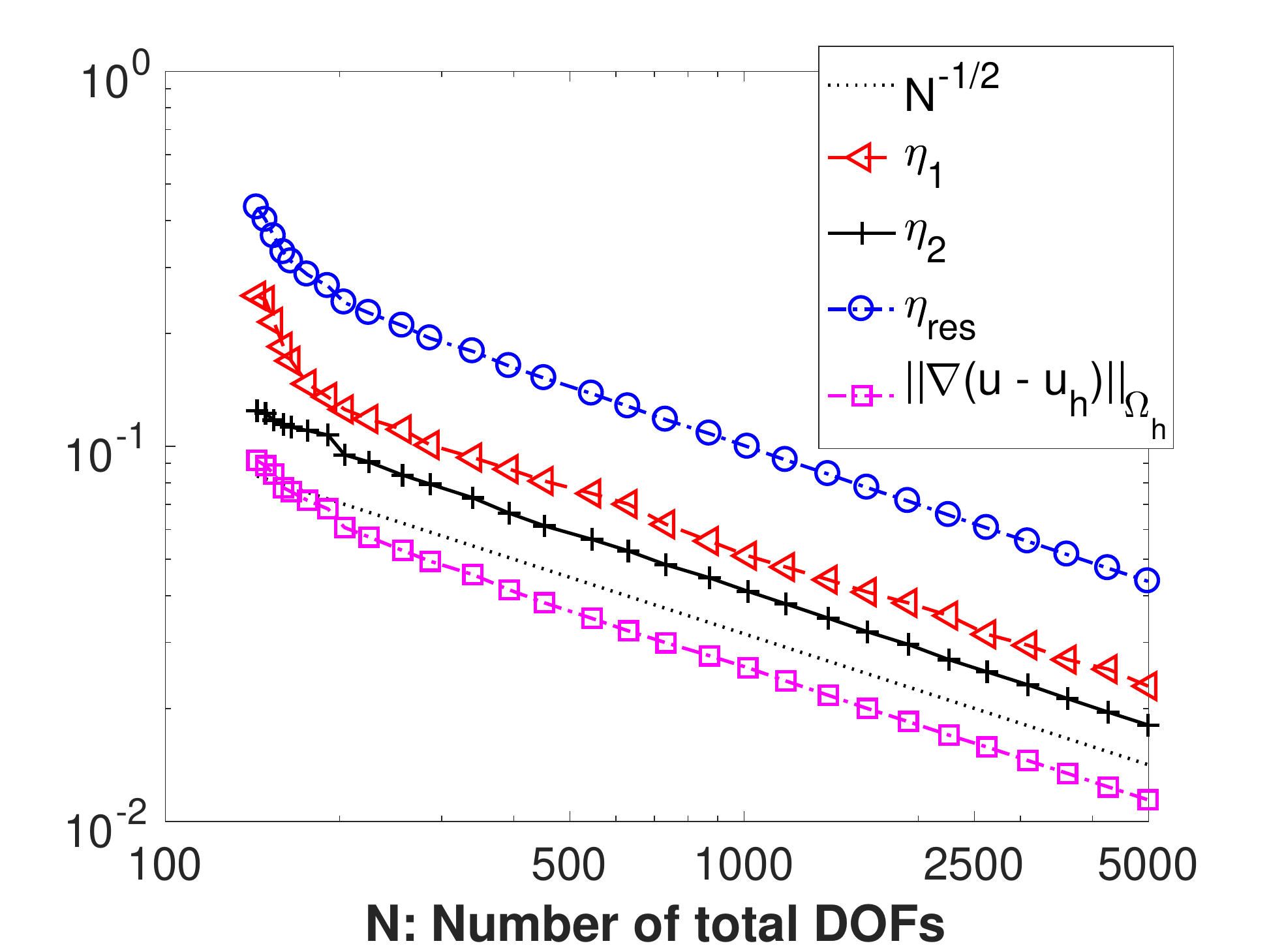}}&
{\includegraphics[width=0.30\textwidth]{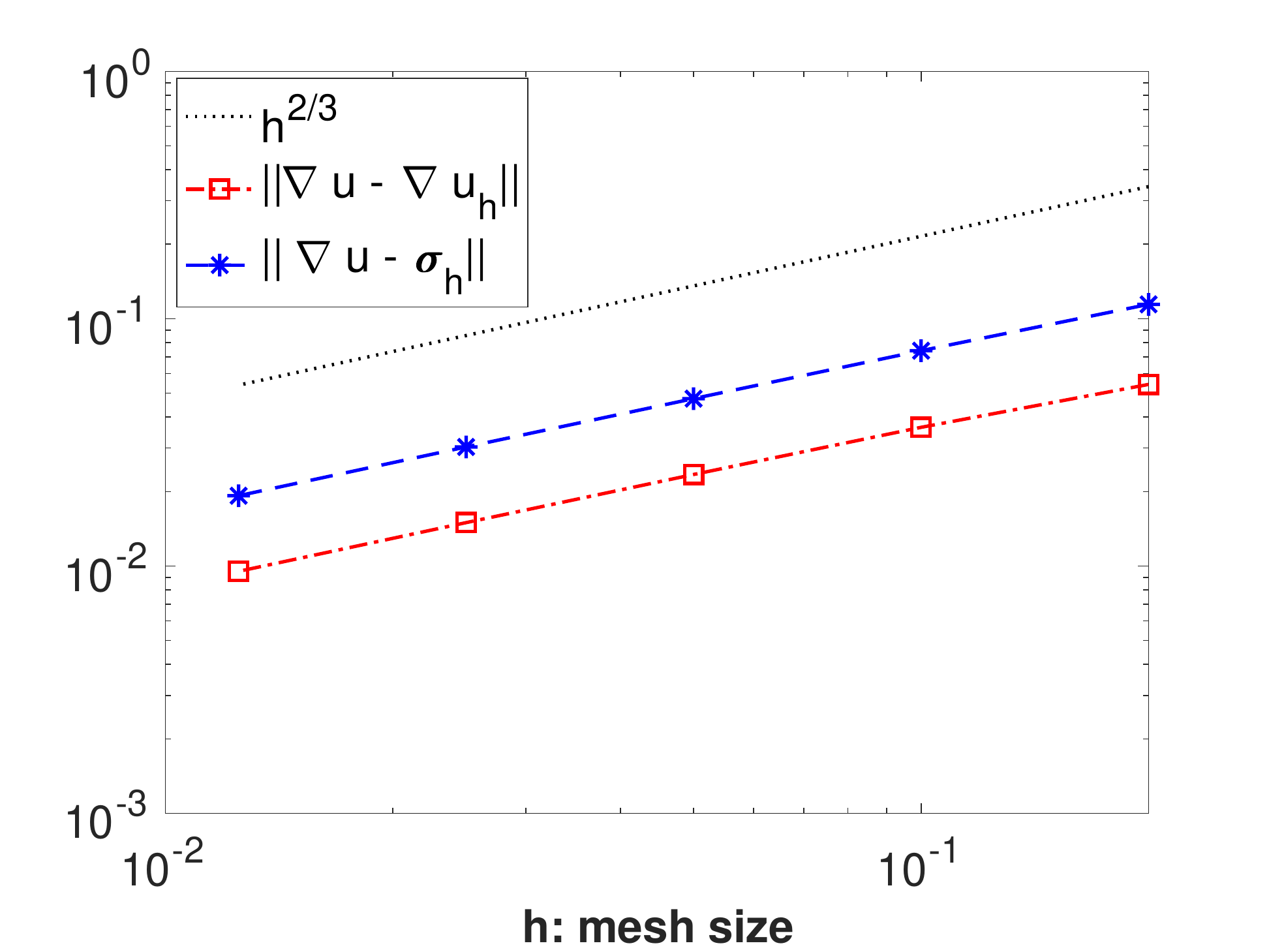}}\\
(c)Errors by $\eta_{2,K}$ &(d)Errors by $\eta_{res, K}$& (e)Errors on uniform meshes\\
\end{tabular}
\caption{\cref{ex3}. Final meshes and convergence of error estimators} %using $\eta_{classical}$ and $\eta$
 \label{Ex3}
\end{figure}

\bibliographystyle{siamplain}
\bibliography{references}
\end{document}